\documentclass[preprint,12pt]{elsarticle}
\usepackage{amssymb}
\usepackage{amsthm}

\usepackage{algorithm}
\usepackage{algpseudocode}
\usepackage{natbib}
\usepackage{mathrsfs}
\usepackage{xcolor} % Required for color customization
\usepackage{booktabs}
\usepackage{rotating}
\usepackage{siunitx}
\usepackage{caption, subcaption}

\usepackage{amsmath}
\usepackage{mathtools}
\mathtoolsset{showonlyrefs}
\usepackage[utf8]{inputenc}
\usepackage{enumitem}
\newlist{steps}{enumerate}{1}
\setlist[steps, 1]{label = \underline{Step \arabic*}.}

\newtheorem{theorem}{Theorem}
\newtheorem{remark}{Remark}

\newtheorem{lemma}{Lemma}

\newtheorem{assumption}{Assumption}
\newtheorem{corollary}{Corollary}

\def\R{\mathbb{R}}
\journal{Mathematics and Computers in Simulation}

\usepackage[colorlinks=true,linkcolor=blue,citecolor=blue,urlcolor=blue]{hyperref}
\usepackage{cancel}

\begin{document}
\sloppy

\begin{frontmatter}



%\title{Convergence of the deep SMP-BSDE for solving stochastic control problems}

\title{Convergence of the deep BSDE method for stochastic control problems formulated through the stochastic maximum principle}


\author[label1]{Zhipeng Huang}
\author[label2]{Balint Negyesi}
\author[label1]{Cornelis W. Oosterlee}
\affiliation[label1]
            {organization={Mathematical Institute, Utrecht University} }
            
\affiliation[label2]
            {organization={Delft Institute of Applied Mathematics (DIAM), Delft University of Technology} }

\begin{abstract}
It is well-known that decision-making problems from stochastic control can be formulated by means of a forward-backward stochastic differential equation (FBSDE). Recently, the authors of \cite{ji2022} proposed an efficient deep learning algorithm based on the stochastic maximum principle (SMP). In this paper, we provide a convergence result for this deep SMP-BSDE algorithm and compare its performance with other existing methods. In particular, by adopting a strategy as in \cite{hanlong2020}, we derive \emph{a-posteriori estimate}, and show that the total approximation error can be bounded by the value of the loss functional and the discretization error. We present numerical examples for high-dimensional stochastic control problems, both in case of drift- and diffusion control, which showcase superior performance compared to existing algorithms.

\end{abstract}


% \begin{highlights}
% \end{highlights}


\end{frontmatter}


\section{Introduction}\label{sec1}

Stochastic control theory is a powerful paradigm for modelling and analyzing decision-making problems that are subject to some random dynamics. Classical approaches for solving these kinds of problems include methods based on the dynamic programming principle (DP) \cite{bellman1958dynamic}, the stochastic maximum principle (SMP) \cite{bismut1978introductory} \cite{pontryagin2018mathematical} and  other techniques, see e.g. \cite{kushner1990numerical} \cite{krylov2004rate} \cite{dong2007rate} \cite{jakobsen2003rate}. However, these approaches cannot easily handle high-dimensional problems and suffer from the “curse of dimensionality". A recent candidate solution technique for stochastic control problems is formed by deep learning-based approaches, due to their remarkable performance in high-dimensional settings. The paper \cite{han2016} developed a deep learning algorithm that directly approximates the optimal control process by a neural network at each step in time, and by training a terminal loss functional for all time steps simultaneously. Similar approaches have been explored in the control theory community \cite{hunt1992neural}, \cite{lehalle1998piecewise} and \cite{psaltis1988}, before the rise of computing power and machine learning.

Inspired by the remarkable performance in \cite{han2016}, the research community has developed several neural network-based algorithms for stochastic control problems, see e.g. \cite{bachouch2022} \cite{hure2021deep} \cite{pereira2020feynman}. Many of these algorithms build upon deriving a forward-backward stochastic differential equation (FBSDE), associated with the control problem. Pioneered by the well-known deep BSDE method, initially proposed by \cite{han2018} and later extended by \cite{hanlong2020} to the coupled setting, several solution approaches have been proposed, showing outstanding empirical performance in high-dimensional frameworks. 
These methods mainly derive the FBSDE through a stochastic representation of the solution of the Hamilton-Jacobi-Bellman (HJB) equation, motivated by the non-linear Feyman-Kac formula. However, such techniques become infeasible when the diffusion of the state process is also controlled, as they do not solve for the value function's second derivative, which is necessary to compute the optimal (diffusion) control. 
As a remedy, enabling diffusion control, the authors in \cite{ji2022} proposed a deep BSDE algorithm where the associated FBSDE is derived from the stochastic maximum principle (SMP), which we call the deep SMP-BSDE. Other SMP-based algorithms include \cite{fouque2020deep} \cite{carmona2022convergence} in the context of mean-field control and mean-field games. We refer to \cite{hu2023recent} and \cite{germain2021neural} for a detailed overview of deep learning algorithms for stochastic control problems.

Furthermore, the authors in \cite{andersson2023} found that deep BSDE methods may fail to converge for FBSDEs stemming from stochastic control problems via DP, due to local minima. They proposed a robust counterpart by adding a regularization component to the loss function which resolved this issue in the case of drift control.
Despite the fact that there are many studies concerning stochastic control, only a few theoretical derivations are available regarding the convergence of machine learning-based approaches for FBSDEs stemming from stochastic control. Theoretical work in this direction includes the study of convergence of the deep BSDE method by \cite{hanlong2020}, which provides \emph{a-posteriori estimate},  and a non-Lipschitz counterpart by \cite{Jiang2021} that only allows the diffusion coefficient to be non-Lipschitz and independent of the BSDE.  The authors of \cite{andersson2023} prove the convergence of a robust deep BSDE method by exploiting the special structure of the FBSDE.  In more recent research, the authors in \cite{reisinger2023} proposed an efficient and reliable \emph{a-posteriori estimate} for fully coupled McKean-Vlasov FBSDEs, which naturally extended the error estimator for decoupled FBSDEs studied in \cite{bender2013}.  For other works in the decoupled framework, see e.g. \cite{hure2020deep}, \cite{negyesi2021one} and \cite{gao2023convergence}.

In this paper, we provide convergence results for the deep SMP-BSDE algorithm and compare the method with existing algorithms supported by numerical results. Unlike \cite{han2018} and \cite{andersson2023}, we consider FBSDEs that come from the SMP, instead of the HJB equation, and therefore the results and standard estimates for FBSDE stemming from DP can not be directly applied. Nevertheless, with some extra effort, we are able to adopt a similar strategy as in \cite{hanlong2020}, and derive the \emph{a-posteriori estimate} for the numerical solutions of the deep SMP-BSDE algorithm  in a multi-dimensional setting  . By this, we are able to tackle diffusion control problems.

This paper is organized as follows: In Section \ref{sec2}, we briefly review the theoretical foundations related to the SMP. In Section \ref{sec3}, we formulate the deep SMP-BSDE algorithm for diffusion control problems. We carry out a convergence analysis in Section \ref{sec4}, and, in particular, \emph{a-posteriori estimate} is derived for the deep SMP-BSDE algorithm. In Section \ref{sec5}, we demonstrate the performance of the algorithm through numerical examples both in the case of drift- and diffusion control.

\section{Background}\label{sec2}

In this section, we review some basic results from stochastic control theory and show how to reformulate a stochastic control problem into an FBSDE through the SMP. 

Let $\left(\Omega, \mathcal{F},\left\{\mathcal{F}_t\right\}_{0 \leq t \leq T}, \mathbb{P}\right)$ be a filtered probability space, supporting an $m$-dimensional Brownian motion $W$ and its natural filtration $\mathcal{F}=\left\{\mathcal{F}_t\right\}_{0 \leq t \leq T}$, augmented by all $\mathbb{P}$-null sets. Fixing $0<T<\infty$, we consider the following finite horizon stochastic control problem
\begin{equation}\label{stochastic_control_problem}
\left\{
\begin{aligned}
    & \inf_{u\in \mathcal{U}[0, T]} J\left( 0, x_0 ; u(\cdot) \right) := \inf_{u\in \mathcal{U}[0, T]} \mathbb{E} \left( \int_{0}^{T} \Bar{f}(s, X_s, u_s) \,ds + g(X_T) \right), \\
	& X_t = x_0 + \int_0^t \Bar{b}(s, X_s, u_s) \,ds + \int_0^t \Bar{\sigma}(s, X_s, u_s) \,dW_s, \quad t \in[0, T],
\end{aligned}
\right.
\end{equation}
where $\bar{b}: [0, T]\times \R^{d}\times \R^\ell\to \R^d$, $\bar{\sigma}: [0, T]\times \R^d\times \R^\ell\to \R^{d\times m}$, $\bar{f}: [0, T]\times \R^d\times\R^\ell\to \R$ and $g: \R^d\to\R$ are deterministic functions, and $X_t, u_t$ are $\R^d, \R^\ell$-valued stochastic processes, respectively. The set of admissible controls, $\mathcal{U}[0,T]$, is defined as
$$
\mathcal{U}[0, T] \coloneqq \left\{u:[0, T] \times \Omega \rightarrow U \mid u \in L_{\mathcal{F}}^2\left(0, T ; \mathbb{R}^\ell \right)\right\},
$$
with
$$
L_{\mathcal{F}}^2\left(0, T ; \mathbb{R}^\ell \right)  
\coloneqq  \left\{x:[0, T] \times \Omega \rightarrow \mathbb{R}^\ell 
\mid x \text{ is } \mathcal{F} \text {-adapted and } E \left[\int_0^T\left\|x_t\right\|^2 \,dt\right]<\infty\right\},
$$
where we shall denote $\| \cdot \|$ for both the usual Euclidean norm and the Frobenius norm for matrices. We assume that the control domain $U$ is a convex body in $\mathbb{R}^\ell$.

Any process $u_t \in \mathcal{U}[0, T]$ is called an admissible control of \eqref{stochastic_control_problem}, and the $(X_t, u_t)$ consisting of the corresponding state process is called an admissible pair. Furthermore, $(X_t, u_t)$ is an optimal pair whenever the infimum of \eqref{stochastic_control_problem} is achieved, and accordingly, we define the value function $V(\cdot, \cdot)$ of \eqref{stochastic_control_problem} as follows
\begin{equation}\label{def:value_function}
\left\{
\begin{aligned}
& V(t, x) \coloneqq \inf_{u \in \mathcal{U}[t, T]} J(t, x ; u(\cdot)),  \quad \forall(t, x) \in[0, T) \times \mathbb{R}^d,\\ 
& V(T, y) \coloneqq g(y), \quad \forall y \in \mathbb{R}^d,
\end{aligned}
\right.
\end{equation}
and denote its partial derivatives as $V_x\coloneqq \partial_x V(t, x)$, $V_{xx}\coloneqq \partial_{xx} V(t, x)$ and $\partial_t V=\partial_t V(t, x)$ without specifying their dependencies in $(t,x)$.

\begin{remark} We shall further distinguish two important classes of problem \eqref{stochastic_control_problem}. We call \eqref{stochastic_control_problem} a drift control problem if the drift coefficient depends on the control $u_t$ but not the diffusion coefficient, and it is called a diffusion control problem if the diffusion coefficient also includes $u_t$ as an argument.
\end{remark}

Associated with the stochastic control problem \eqref{stochastic_control_problem}, we introduce the adjoint equation, as follows
\begin{equation}\label{eq:adjoint}
\left\{
\begin{aligned}
    & P_t = P_0 - \int_0^t \nabla_x \Bar{H}(s, X_s, u_s, P_s, Q_s) \,ds + \int_0^t Q_s\, d W_s, \quad t \in[0, T] \\
    & P_T = - \nabla_x g(X_T),
\end{aligned}
\right.
\end{equation}
where the $\nabla_x  \Bar{H}$ is the derivative of Hamiltonian $ \Bar{H}$, which is defined by
\begin{equation}\label{def:Ham}
\begin{aligned}
& \Bar{H}(t, x, u, p, q) := p^\top \Bar{b}(t, x, u) + \operatorname{Tr}\left(q^{\top} \Bar{\sigma}(t, x, u)\right) - \Bar{f}(t, x, u),\\
&\quad \forall (t, x, u, p, q) \in[0, T] \times \mathbb{R}^d \times U \times \mathbb{R}^d \times \mathbb{R}^{d \times m}.
\end{aligned}
\end{equation}
Equation \eqref{eq:adjoint} is a BSDE whose solution is formed by a pair of processes, $(P(\cdot), Q(\cdot)) \in L_{\mathcal{F}}^2\left(0, T ; \mathbb{R}^d\right) \times \left(L_{\mathcal{F}}^2\left(0, T ; \mathbb{R}^d\right)\right)^m$.
Concerning the wellposedness of BSDE \eqref{eq:adjoint}, we state the following assumption first.

\begin{assumption}\label{assume_adjoint}\
Let $\bar{\varphi} = \bar{b}, \bar{\sigma}, \bar{f}$ and $g$. The map $\bar{\varphi}$ is $C^2$ in $x$, and $\bar{\varphi}(t, 0, u)$ is bounded for any $(t,u)\in [0,T] \times U$. Moreover, $\bar{\varphi}$, $\bar{\varphi}_x$ and $\bar{\varphi}_{xx}$ are uniformly Lipschitz in $x$ and $u$.
\end{assumption}

\begin{remark}\label{wellpose_adjoint} With Assumption \ref{assume_adjoint}, the adjoint equation, or BSDE \eqref{eq:adjoint}, admits a unique solution for every admissible pair $(X_t, u_t)$. However, there is only one of the admissible 4-tuples $(X_t, u_t, P_t, Q_t)$ that also minimizes the objective function in \eqref{stochastic_control_problem}. Therefore, to have an FBSDE that admits a unique solution without including the objective function of the control problem, we need extra effort in the reformulation. For this purpose, we recall the SMP below.
\end{remark}

\begin{theorem}[Stochastic maximum principle]\label{thm:SMP}
Let Assumption \ref{assume_adjoint} hold, and $\left( X_t^*, u_t^*, P_t^*, Q_t^* \right)$ be an admissible 4-tuple. Suppose that $g(\cdot)$ is convex, $\Bar{H}\left(t, \cdot, \cdot, P_t^*, Q_t^* \right)$ defined by (\ref{def:Ham}) is concave for all $t \in[0, T]$ $\mathbb{P}$ almost surely, and the maximum condition
\begin{equation}\label{maxcond}
\Bar{H} \left( t , X_t^*, u_t^*, P_t^*, Q_t^* \right) = \max_{u \in U} \Bar{H} \left( t , X_t^*, u, P_t^*, Q_t^* \right), \quad \text{a.e.} \ t \in[0, T], \quad \mathbb{P} \text{-a.s.}
\end{equation}
holds. Then, $\left( X_t^*, u_t^* \right)$ is an optimal pair of problem \eqref{stochastic_control_problem}.
\end{theorem}

\begin{remark}
The proof of this theorem is found in \cite[pp. 149-150]{pham2009continuous}, or, in a more general setting, \cite[pp. 138-140]{yong1999stochastic}. Theorem \ref{thm:SMP} provides a sufficient condition for the optimal control $u^*_t$, when certain concavity and convexity conditions hold, which are crucial in general, see for instance, \cite[Example 3.1, pp. 138-140]{yong1999stochastic}. On the other hand, Theorem \ref{thm:SMP} itself does not constitute a necessary condition unless there is no diffusion control in problem \eqref{stochastic_control_problem}. In the rest of this paper, we shall use superscript $*$ to indicate that a process is associated with the optimal control $u_t^*$  whenever it needs to be further distinguished. 
\end{remark}

Under sufficient smoothness and concavity assumptions of $\bar{H}$ in Theorem \ref{thm:SMP}, the optimization \eqref{maxcond} is uniquely solved by the following first-order conditions,
\begin{align}\label{feedback:first-order-condition}
\nabla_u \bar{H}(t, X^*_t, u^*_t, P^*_t, Q^*_t) 
= &(\nabla_u \bar{b}(t, X^*_t, u^*_t))^\top P^*_t \\
& + \nabla_u \operatorname{Tr}(\bar{\sigma}^\top(t, X^*_t, u^*_t) Q^*_t) - \nabla_u\bar{f}(t, X^*_t, u^*_t)=0.  
\end{align}
 Under the setting of Algorithm 3 in \cite{ji2022}, without loss of generality  we assume that the first-order conditions yields an explicit formula for the mapping $\mathcal{M}: (t, X_t^*, P_t^*, Q_t^*)\mapsto u_t^*$,
\begin{equation}\label{feedback}
    u_t^* = \mathcal{M}(t, X^*_t, P^*_t, Q^*_t), \quad \forall t \in [0, T].
\end{equation}
We will call this function the \emph{feedback map}, and remark that for a rather wide range of interesting problems such an expression is available in closed-form.
 
Let us define the map $\varphi = b, \sigma, f$ and $g$, which is given by the composition $\varphi \coloneqq \bar{\varphi}(t, x, \mathcal{M}(t,x,p,q))$ for $\bar{\varphi} = \bar{b}, \bar{\sigma}$ and $\bar{f}$. 

Similarly, we define the function $\bar{F}(t,x,u,p,q) \coloneqq \nabla_x \bar{H}(t,x,u,p, q)$ and write
\begin{equation}\label{def:Ham_pq}
F(t, x, p, q) 
\coloneqq \bar{F}(t,x, \mathcal{M}(t,x,p,q), p, q).
\end{equation}
With this in hand, we reformulate \eqref{eq:adjoint} and the controlled SDE of \eqref{stochastic_control_problem} as a fully-coupled FBSDE, for $t\in[0, T]$,
\begin{equation}\label{eq:FBSDE}
\left\{
\begin{aligned}
X_t &= x_0 + \int_0^t b(s, X_s, P_s, Q_s) \,ds + \int_0^t \sigma(s, X_s, P_s, Q_s) \,dW_s, \\
P_t &= P_0 - \int_0^t F(s, X_s, P_s, Q_s) \,ds + \int_0^t Q_s\, d W_s, \\
P_T &= - \nabla_x g(X_T). 
\end{aligned}
\right.
\end{equation}
We call the BSDE  part in \eqref{eq:FBSDE}, subject to its boundary condition, the SMP-BSDE.

\begin{remark}\label{remark:4} An important feature of \eqref{eq:FBSDE} is that its solution is equivalent to the solution of \eqref{stochastic_control_problem}. Let $(X^*_t, u^*_t)$ be an optimal pair obtained by \eqref{stochastic_control_problem}, then there exists a unique pair $(P^*_t, Q^*_t)$ which follows from Remark \ref{wellpose_adjoint}. Therefore, we have the optimal $(X^*_t, P^*_t, Q^*_t)$ that also solves \eqref{eq:FBSDE}. On the other hand, solving \eqref{eq:FBSDE} gives us $(\tilde{X}_t, \tilde{P}_t, \tilde{Q}_t )$, and we obtain $\tilde{u}_t$ by the feedback map \eqref{feedback}. Consequently, we must have $\tilde{u}_t = u^*_t$ and $(\tilde{X}_t, \tilde{P}_t, \tilde{Q}_t ) = (X^*_t, P^*_t, Q^*_t)$, due to the uniqueness of feedback map \eqref{feedback} and Theorem \ref{thm:SMP}.
\end{remark}

\section{Existing algorithms and the deep SMP-BSDE algorithm}\label{sec3}

In this section, we first give a comparison between the existing deep BSDE algorithm and the deep SMP-BSDE algorithm, and show why the former can not be used to solve diffusion control problems. 

Suppose the DP holds, then the value function $V(t,x)$ of problem \eqref{stochastic_control_problem} solves the following HJB equation
\begin{equation}\label{eq:HJB}
\left\{
\begin{aligned} 
& - \partial_t V + \sup_{u \in U} \mathcal{G}\left(t, x, u, -V_x, -V_{x x}\right)=0,\\
& V(T, x) = g(x),
\end{aligned}
\right.
\end{equation}
where the generalized Hamiltonian $\mathcal{G}$ is defined by
\begin{equation}\label{generazlied_Hamiltonian}
\mathcal{G}\left(t, x, u, -V_x, -V_{xx}\right) \coloneqq  
- \frac{1}{2} \operatorname{Tr}\left(\bar{\sigma}(t, x, u)^{\top} V_{xx} \bar{\sigma}(t, x, u)\right) - V_x^\top \bar{b}(t, x, u) - \bar{f}(t, x, u)
\end{equation}
By the stochastic verification theorem, see \cite[pp. 268-269]{yong1999stochastic}, a given admissible pair $(X_t^*, u_t^*)$ is optimal if and only if 
\begin{equation}
\begin{aligned}
\mathcal{G} \left(t, X_t^*, u_t^*, -V_x, -V_{xx}\right)=\max_{u \in U} \mathcal{G} \left(t, X_t^*, u,-V_x,-V_{xx}\right),
\end{aligned}
\end{equation}
provided that $V$ is a classical solution to \eqref{eq:HJB} with sufficient smoothness.

Suppose \eqref{eq:HJB} admits a unique classical solution given by the value function $V(t,x)$. Then, it is straightforward to obtain a new feedback map $\tilde{\mathcal{M}}$ by maximizing $\mathcal{G}$ over $u$, which allows us to write the optimal control $u^*_t$ in terms of $X^*_t$, $V_x$ and $V_{xx}$
$$ u^*_t = \tilde{\mathcal{M}}(t, X^*_t, -V_x, -V_{xx} ) , \qquad  t\in [0, T].$$
Notice that whenever the diffusion includes the control variable, $\tilde{\mathcal{M}}$ will always depend on $V_{xx}$. From the stochastic representation of $V(t,x)$, we obtain for $t\in[0,T]$
\begin{equation}\label{valuefunction_BSDE}
V(t, x) = g(X^*_T) + \int_t^T \bar{f}(s, X^*_s, u^*_s) \,ds - \int_t^T V_x^\top \bar{\sigma}(t, X_t^*, u_t^*) \, d W_s .
\end{equation}
Equation \eqref{valuefunction_BSDE} defines a BSDE whose unique solution coincides with $(Y_t, Z_t) = \left(V(t, X_t^*),  V_x^\top\bar{\sigma}(t, X^*_t, u^*_t) \right)$, whenever the feedback map $\tilde{\mathcal{M}}$ is only a function of $V_x$ but not of $V_{xx}$. However, this is only the case when the diffusion coefficient does not depend on $u$, i.e., in the case of drift control. In particular, equation \eqref{valuefunction_BSDE} coincides with the deep BSDE method of \cite{han2018} and \cite{andersson2023}, applied to stochastic control problems. In numerical settings where one does not have direct access to $V_{xx}$, this makes the dynamic programming approaches described above infeasible, whenever the diffusion coefficient depends on $u$.

Moreover, the robust counterpart of the deep BSDE method in \cite{andersson2023}, which first simulates $X_t$ forward in time and computes the samples of $Y_0$, denoted by $\mathcal{Y}_0$, backwards in time, according to \eqref{valuefunction_BSDE}, introduces an objective functional of the form
\begin{equation}\label{robust_obj}
\inf_{\theta^Z} E\left( \mathcal{Y}_0 \right) + \lambda \operatorname{Var}\left( \mathcal{Y}_0 \right)
\end{equation}
for some parametric space $\theta^Z$ of the neural networks for approximating $Z_t$ and a chosen constant $\lambda>0$. Such formulation builds upon the facts that the corresponding DP approach results in a BSDE whose driver is independent of $Y_t=V(t, X_t^*)$, and $Y_0 = V(t_0, x_0)$ is a deterministic quantity coinciding with the value function of the control problem. Moreover, $\operatorname{Var}\left( \mathcal{Y}_0 \right)$ is shown to be equal to the terminal condition of the BSDE, see \cite[sec. 3]{andersson2023}, and therefore \eqref{robust_obj} should be minimized. However, this robustness technique does not apply to the SMP-BSDE setting, as in general $F$ in \eqref{def:Ham_pq} depends on $P_t$, and since $P_0 = -V_x(t_0, x_0)$ its mean is not necessarily minimized at $t_0$. From this point of view, the deep SMP-BSDE algorithm seems particularly promising for solving general stochastic control problems.

For the reasons above, in order to be able to treat diffusion control problems, in what follows we focus on the SMP-based FBSDE formulation. Let us state the discrete scheme of the deep SMP-BSDE algorithm originating from Algorithm 3 of \cite{ji2022}, but instead of using one single neural network for both $P_0$ and the process $Q_t$, as in the aforementioned paper, we approximate $P_0$ and $Q_t$ at each step in time by a separate neural network, respectively.
We consider the following classical Euler scheme corresponding to \eqref{eq:FBSDE}
\begin{subequations}
    \begin{align}
        \inf_{\mu_0^\pi \in \theta_0^P, \phi_i^\pi \in  \theta_i^Q}  &E \left\| - \nabla_x g\left(X_T^\pi\right) - P_T^\pi \right\|^2, \label{euler_objective}  \\
        &\text{s.t.}
\left\{
\begin{aligned}
&X_0^\pi= x_0, \\
&P_0^\pi= \mu_0^\pi(x_0), \\
&X_{t_{i+1}}^\pi = \begin{aligned}[t]
    X_{t_i}^\pi &+ b\left(t_i, X_{t_i}^\pi, P_{t_i}^\pi, Q_{t_i}^\pi \right) \Delta t\\
    &+ \sigma\left(t_i, X_{t_i}^\pi, P_{t_i}^\pi, Q_{t_i}^\pi \right) \Delta W_i,
\end{aligned} \\
& Q_{t_i}^\pi = \phi_i^\pi\left(X_{t_i}^\pi \right), \\
& P_{t_{i+1}}^\pi = P_{t_i}^\pi - F\left(t_i, X_{t_i}^\pi, P_{t_i}^\pi, Q_{t_i}^\pi\right) \Delta t + Q_{t_i}^\pi \Delta W_i,
\end{aligned}
\right.\label{eq:euler}
    \end{align}
\end{subequations}
% \noeqref{euler_objective}
with a time partition, $\pi: 0=t_0<t_1<\cdots<t_N=T$, $h=T / N$, 
$t_i=i h$ and $\Delta W_i:=W_{t_{i+1}}-W_{t_i}$ for $i=0,1, \ldots, N-1$. We recall the definitions of $b$, $\sigma$ and $F$ in Section \ref{sec2}. Moreover, we let $\theta_0^P$ and $\theta_i^Q$ be the corresponding parametric spaces for the neural networks $\theta_0^P\ni\mu_0^\pi: \R^d\to \R^d$ and $\theta_i^Q\ni\phi_i^\pi:\R^d\to\R^{d\times m}$, respectively. The objective function \eqref{euler_objective} serves as the loss function in the machine learning algorithm, and through the training of the neural networks we wish to find appropriate functions $\mu_0^\pi(x_0)$ and $\phi_i^\pi\left(X_{t_i}^\pi \right)$ that can approximate $P_0$ and $Q_{t_i}$ sufficiently well.
The complete pseudo-code for the deep SMP-BSDE algorithm is given in Algorithm \ref{algorithm}, which will be used in the later Section \ref{sec5} for numerical experiments.
\begin{algorithm}[H]
  \caption{deep SMP-BSDE algorithm}
  \begin{algorithmic}[1] 
    \State \textbf{Input:} Initial parameters $\left(\theta^P_0, \theta^Q_0,\ldots, \theta^Q_{N-1}\right)$, learning rate $\eta$; batch size $M$; number of iteration $K$.
    \State \textbf{Data:} Simulated Brownian increments $\left\{ \Delta W_{t_i, k} \right\}_{0\leq i\leq N-1, 1\leq k\leq K}$
    \State \textbf{Output:} The triple $(X_{t_i}, P_{t_i}, Q_{t_i})$
    \For{$k = 1$ to $K$}
        \State $X_{t_0, k}^{\pi} = x_0$, $P_{t_0, k}^{\pi} = \mu_0^\pi(x_0; \theta^P_{0} ) $ 
        \For{$i = 0$ to $N-1$}
        \State $ Q_{t_i, k}^{\pi} = \phi_i^\pi\left(X_{t_i}^\pi; \theta^Q_{i}\right) $ 
        \State $ u_{t_i, k}^{\pi} = \mathcal{M}\left(t_i, X_{t_i, k}^{\pi}, P_{t_i,k}^{\pi}, Q_{t_i,k}^{\pi} \right)$
        \State $ X_{t_{i+1},k}^{\pi} =  X_{t_i,k}^{\pi} + \bar{b} \left(t_i, X_{t_i,k}^{\pi},u_{t_i, k}^{\pi} \right) \Delta t_i  + \bar{\sigma} \left(t_i, X_{t_i,k}^{\pi}, u_{t_i, k}^{\pi} \right) \Delta W_{t_i, k} $
        \State $ P_{t_{i+1},k}^{\pi} = P_{t_i,k}^{\pi} -  \bar{F} \left(t_i, X_{t_i,k}^{\pi}, u_{t_i, k}^{\pi} \right) \Delta t_i + Q_{t_i,k}^{\pi} \Delta W_{t_i, k} $ 
        \EndFor
    \State $\text{Loss} = \frac{1}{M} \sum_{j=1}^M  \left\| -\nabla_x g\left( X_{t_N,k}^{\pi} \right) - P_{t_N,k}^{\pi}\right\|^2 $
    \State $  \left(\theta^P_0, \theta^Q_0,\ldots, \theta^Q_{N-1} \right)  \longleftarrow  \left(\theta^P_0, \theta^Q_0,\ldots, \theta^Q_{N-1} \right) - \eta \nabla \text{Loss} $
    \EndFor
  \end{algorithmic}
  \label{algorithm}
\end{algorithm}

\section{Convergence Analysis}\label{sec4}

This section is dedicated to the convergence analysis for the deep SMP-BSDE method, reviewed in Section \ref{sec2}, and the discrete scheme \eqref{eq:euler}. In particular, we show that the total error of the numerical solution to the FBSDE is bounded by the time discretization error and the simulation error of the objective function, and such error in theory could be made arbitrarily small in a sufficiently fine grid, due to the universal approximation theorem. Our analysis follows a similar strategy as \cite{hanlong2020}. For the sake of this section, in order to be able to use techniques established therein, we need the following restriction.
\begin{assumption}\label{assumption:drift_control}
    The coefficients $\bar{b}, \bar{\sigma}$ in \eqref{stochastic_control_problem} imply a feedback map such that the compositions $\bar{b}(t, x, \mathcal{M}(t, x, p, q))$ and $\bar{\sigma}(t, x, \mathcal{M}(t, x, p, q))$ do not depend on $q$.
\end{assumption}
Even though this condition may seem abstract, note that it in particular covers the important subclass of drift control problems, i.e. whenever $\bar{\sigma}$ does not depend on $u$. 
In that case, under the convexity conditions of Theorem \ref{thm:SMP}, the feedback map in \eqref{feedback} is only a function of $(t, X_t, P_t)$ and not of $Q_t$ -- see the first-order conditions in \eqref{feedback:first-order-condition}. Consequently, the solution pair of the backward equation in the FBSDE \eqref{eq:FBSDE} only couples into the forward component via $P_t$, similarly as in \cite{hanlong2020}. Nevertheless, in what follows we follow a more general presentation with Assumption \ref{assumption:drift_control}, such that special cases of diffusion control problems can also be treated.

Our main contribution is establishing a convergence result for the numerical solution of the stochastic control problem \eqref{stochastic_control_problem} by Algorithm \ref{algorithm}, where the corresponding FBSDE is vector-valued and derived from the SMP. In this regard, our results can be viewed as an extension to \cite{hanlong2020}, where the backward equation in \eqref{eq:FBSDE} was only scalar-valued. 

We acknowledge that a recent article \cite{reisinger2023} has studied the convergence for a more general problem formulation, i.e., a fully coupled McKean-Vlasov FBSDE (MV-FBSDE), where the maps $b$, $\sigma$ and $F$ may also depend on the $Q_t$ process and the law of the solution.  The main difference between these two works is the set of assumptions and the corresponding methods for studying the well-posedness of the discretization of the equation. To be specific, \cite{reisinger2023} have studied the well-posedness and stability of an implicit Euler discretization of the MV-FBSDE, and consequently adapted the method of continuation for studying this discretization, which requires a structural monotonicity assumption about the MV-FBSDE. On the contrary, our work utilizes the well-posedness and convergence of the implicit scheme from \cite{bender2008time}, which is developed through a fixed-point argument and uses a different set of weak coupling and monotonicity conditions formally stated in \cite{hanlong2020, bender2008time}. Indeed, by adopting the structural monotonicity assumptions that are also used in \cite{BENSOUSSAN2015, peng1999}, the authors of \cite{reisinger2023} study a more general error estimator compared to the one in \cite{hanlong2020}, and extends the work \cite{bender2013} to a coupled MV-FBSDE setting.  On the other hand, our work is an extension of \cite{hanlong2020} and a convergence study of the algorithm proposed by \cite{ji2022}, building on a different set of weak coupling and monotonicity conditions, and therefore these two works do not contain each other. 

To conduct our analysis, we require the following technical, standing assumptions to hold.

\begin{assumption}\label{assumption:holder} Suppose
\begin{enumerate}
    \item The maps $\bar{b}, \bar{\sigma}, \bar{F}$ are uniformly Hölder- $\frac{1}{2}$  continuous in $t$. 
    \item The feedback map $\mathcal{M}(t,x,p, q)$ is uniformly Hölder- $\frac{1}{2}$  continuous in $t$, and Lipschitz continuous in $x$, $p$ and $q$, respectively.
\end{enumerate}
\end{assumption}

\begin{assumption}\label{assumption:kb_kf}
There exist constants $k^b$ and $k^F$, that are possibly negative, such that
\begin{equation}
\begin{aligned}
\left( b(t, x_1, p) - b(t, x_2, p) \right)^\top \Delta x 
& \leq k^b \|\Delta x\|^2, \\
\left( F(t, x, p_1, q)-F(t, x, p_2, q) \right)^\top \Delta p 
& \leq k^F \|\Delta p\|^2.
\end{aligned}
\end{equation}
\end{assumption}

\begin{remark}\label{lipschitz_constants} With Assumptions \ref{assume_adjoint}, \ref{assumption:drift_control} and \ref{assumption:holder}, this implies that $b, \sigma, F$ and $g$ are uniformly Lipschitz continuous in all spatial variables, and therefore we may write
\begin{equation}
\begin{aligned}
\left\|b(t, x_1, p_1)-b(t, x_2, p_2)\right\|^2 
& \leq L^{b}_x \|\Delta x\|^2+ L^{b}_p \|\Delta p\|^2 ,  \\
\left\|\sigma(t, x_1, p_1) - \sigma(t, x_2, p_2)\right\|^2 
& \leq L^{\sigma}_x \|\Delta x\|^2 + L^{\sigma}_p \|\Delta p\|^2, \\
\left\| F\left(t, x_1, p_1, q_1\right) - F\left(t, x_2, p_2, q_2\right)\right\|^2 
& \leq L^{F}_x \|\Delta x\|^2+ L^{F}_p \|\Delta p\|^2 + L^{F}_q \|\Delta q\|^2, \\
\left\| \nabla_x g (x_1) - \nabla_x g(x_2)\right\|^2 & \leq  L^{g_{x}}_x \|\Delta x\|^2.
\end{aligned}
\end{equation}
Similarly, the Hölder-continuity of $b, F$ and $\sigma$ follow from the same set of assumptions, which imply that $b(t,0,0), F(t,0,0,0)$ and $\sigma(t,0,0)$ are bounded in $t$, and the boundedness of $\nabla_x g(x)$ directly follows from Assumption \ref{assume_adjoint}. For convenience, we use $\mathscr{L}$ to denote the set of all constants mentioned above and denote its upper bound by $L$.
\end{remark}

Next, we introduce the following system of quasi-linear parabolic PDEs, which is naturally associated with \eqref{eq:FBSDE}, 
\begin{equation}\label{pdeforsmp}
\left\{
\begin{aligned}
& \partial_t \nu^i  +\frac{1}{2} \operatorname{Tr} \left( \partial_{xx} \nu^i \sigma \sigma^{\top}(t, x, \nu) \right) + \partial_x \nu^i b\left(t, x, \nu, \partial_x \nu \sigma(t, x, \nu)\right) \\
& \qquad + F^i\left(t, x, \nu, \partial_x \nu \sigma(t, x, \nu)\right)=0, \quad \forall i=1, \cdots, d; \\
& \nu(T, x) = -\nabla_x g(x),
\end{aligned}
\right.
\end{equation}
where $\nu^i$ denotes the $i$-th component of the vector $\nu$.  Note that this is not the same as HJB equation \eqref{eq:HJB}, which is associated with a different FBSDE given by \eqref{valuefunction_BSDE}.

For this system of PDEs, we shall recall the so-called weak and monotonicity conditions studied in \cite{bender2008time, hanlong2020}, as well as in earlier literature \cite{antonelli1993backward, pardoux1999forward}. With these additional conditions, the system of PDEs \eqref{pdeforsmp} admits a unique viscosity solution $\nu$ and there is a unique solution $\left(X_t, P_t, Q_t\right)$ to the FBSDE \eqref{eq:FBSDE}, connected by $\nu\left(t, X_t\right) = P_t$. For a complete statement of these conditions and results, we refer to \cite{bender2008time, hanlong2020}.

\begin{remark}\label{PQ_representation} 
It is worth to mention that with additional stronger assumptions such as smoothness and boundedness conditions, the system of PDEs \eqref{pdeforsmp} admits a unique classical solution, and enjoys the representations $\nu(t, X_t) = P_t$ and $\partial_x \nu(t, X_t) \sigma(t, X_t, \nu(t, X_t) ) = Q_t $ via the nonlinear Feynman-Kac formulae. Moreover, a connection to the derivatives of value function $V(t, X_t^*)$ \eqref{def:value_function} evaluated at the optimal pair $(X_t^*, u_t^*)$ to the control problem can be established by
\begin{equation}\label{PQ_representation_drift}
P_t = - V_x(t, X_t^*), \quad Q_t = - V_{xx}(t, X_t^*) \bar{\sigma}(t, X_t^*, u_t^*),
\end{equation}
that follows directly from \cite[pp. 151-152]{pham2009continuous} or \cite[pp. 250-253]{zhang2017backward}, see also Remark \ref{remark:4}.  
\end{remark}

Although a classical solution enables us to solve control problems whenever both $V_{x}$ and $V_{xx}$ are of interest, e.g. delta-gamma hedging problems in finance, we stick with the viscosity solution setting which requires weaker conditions and is sufficient enough for us to derive the results in this paper. With the viscosity solution $\nu$, also called decoupling field in BSDE literature, one can decouple the FBSDE \eqref{eq:FBSDE} and obtain the following.

\begin{theorem}[Convergence of the implicit scheme]\label{convergence_im}
Suppose Assumptions \ref{assume_adjoint}-\ref{assumption:kb_kf} hold, and furthermore let the weak and monotonicity conditions in \cite{hanlong2020} hold. Then for a sufficiently small $h$, the following discrete-time equation for $0 \leq i \leq N-1$,
\begin{equation}\label{implicit_scheme}
\left\{
\begin{aligned}
&\bar{X}_0^\pi= x_0, \\
&\bar{X}_{t_{i+1}}^\pi=\bar{X}_{t_i}^\pi + b\left(t_i, \bar{X}_{t_i}^\pi, \bar{P}_{t_i}^\pi\right) h + \sigma\left(t_i, \bar{X}_{t_i}^\pi, \bar{P}_{t_i}^\pi\right) \Delta W_i, \\
&\bar{P}_T^\pi = - \nabla_x g\left(\bar{X}_T^\pi\right), \\
&\bar{Q}_{t_i}^\pi=\frac{1}{h} E\left(\bar{P}_{t_{i+1}}^\pi \Delta W_i^\top \mid \mathcal{F}_{t_i}\right), \\
&\bar{P}_{t_i}^\pi = E\left(\bar{P}_{t_{i+1}}^\pi + F\left(t_i, \bar{X}_{t_i}^\pi, \bar{P}_{t_i}^\pi, \bar{Q}_{t_i}^\pi\right) h \mid \mathcal{F}_{t_i}\right),
\end{aligned}
\right.
\end{equation}
has a solution, $\left(\bar{X}_{t_i}^\pi, \bar{P}_{t_i}^\pi, \bar{Q}_{t_i}^\pi\right)$, such that $\bar{X}_{t_i}^\pi \in L^2\left(\Omega, \mathcal{F}_{t_i}, \mathbb{P}\right)$ and
\begin{align}\label{estimate:implicit}
\begin{aligned}[b]
    \sup_{t \in[0, T]}\Big(E\left\|X_t-\bar{X}_t^\pi\right\|^2&+E\left\|P_t-\bar{P}_t^\pi\right\|^2\Big)\\
    &+\int_0^T E\left\|Q_t-\bar{Q}_t^\pi\right\|^2 \, d t
\end{aligned} \leq C\left(1+E\|x_0\|^2\right) h,
\end{align}
where $\bar{X}_t^\pi=\bar{X}_{t_i}^\pi, \bar{P}_t^\pi=\bar{P}_{t_i}^\pi, \bar{Q}_t^\pi=\bar{Q}_{t_i}^\pi$, for $t \in\left[t_i, t_{i+1}\right)$, $\left( X_t, P_t, Q_t \right)$ is the solution to \eqref{eq:FBSDE}, and $C$ is a constant depending on $\mathscr{L}$ and $T$.
\end{theorem}

\begin{remark} We briefly outline how to derive the above theorem using weak and monotonicity conditions and the results of \cite{bender2008time}. First, the existence of the solution $\left(\bar{X}_{t_i}^\pi, \bar{P}_{t_i}^\pi, \bar{Q}_{t_i}^\pi\right)$ is proved by the convergence of the approximated decoupling field using a fixed point argument, see Theorem 5.1 (ii) in \cite{bender2008time}. Then the error estimates \eqref{estimate:implicit} can obtained by Theorem 6.5 in the same literature, which essentially is built upon estimates for the approximated decoupling fields and the true counterpart. 
\end{remark}

\begin{remark} We further remark that we only need the well-posedness and the estimates \eqref{estimate:implicit} for the implicit scheme here, and there exist different sets of conditions such that similar results hold since the weak and monotonicity conditions are sufficient only. For instance, \cite{reisinger2023} derive a similar result for an implicit and forward Euler discretization of a coupled MV-FBSDE, using a different set of monotonicity assumptions and the method of continuation.
\end{remark}

Now, recall the classical Euler scheme \eqref{eq:euler}. For the discrete equation of $P_{t_{i+1}}^\pi$ in \eqref{eq:euler}, we take the conditional expectation, $E\left(\cdot \mid \mathcal{F}_{t_i}\right)$, at both sides and obtain the conditional expectation representation for $P_{t_i}^\pi$. For the same discrete equation of $P_{t_{i+1}}^\pi$, we multiply by $\left(\Delta W_i\right)^\top$ and take $E\left(\cdot \mid \mathcal{F}_{t_i}\right)$ afterwards. Therefore, we obtain a formulation, as follows
\begin{equation}\label{formulation3}
\left\{
\begin{aligned}
& X_0^\pi = x_0, \\
& X_{t_{i+1}}^\pi =X_{t_i}^\pi + b\left(t_i, X_{t_i}^\pi, P_{t_i}^\pi  \right) h+\sigma\left(t_i, X_{t_i}^\pi, P_{t_i}^\pi\right) \Delta W_i, \\
& Q_{t_i}^\pi = \frac{1}{h} E \left( P_{t_{i+1}}^\pi \Delta W_i^\top \mid \mathcal{F}_{t_i}\right), \\
& P_{t_i}^\pi = E\left( P_{t_{i+1}}^\pi + F\left(t_i, X_{t_i}^\pi, P_{t_i}^\pi,  Q_{t_i}^\pi\right) h \mid \mathcal{F}_{t_i}\right) .
\end{aligned}\right.
\end{equation}

\begin{remark} A key feature of this formulation is that it does not include the boundary condition for $P_T^\pi$, and therefore there are infinitely many solutions. In particular, it is easy to see that both the classic Euler scheme \eqref{eq:euler} and the implicit scheme \eqref{implicit_scheme} are solutions to this formulation. Such a feature is particularly suitable for the analysis of the algorithm, as we set our loss function \eqref{euler_objective} to measure the distance between $P_T^\pi$ and $-\nabla_x g(X_T^\pi)$, and one may expect that the closer the two solutions of \eqref{formulation3} are, the smaller loss we will have, and vice versa. 
\end{remark}

In what follows, we shall apply the same techniques in \cite{hanlong2020} to the current vector-valued BSDE setting in order to prove Lemma \ref{lem:estimate1} and Theorem \ref{thm:estimate2}. 

\begin{lemma}\label{lem:estimate1} 
For $j=1,2$, suppose $\big\{(X_{t_{i}}^{\pi, j}, P_{t_{i}}^{\pi, j}, Q_{t_{i}}^{\pi, j})\big\}_{0\leq i\leq N-1}$ are two solution triples of \eqref{formulation3}, with $X_{t_i}^{\pi, j}, P_{t_i}^{\pi, j} \in L^2\left(\Omega, \mathcal{F}_{t_i}, \mathbb{P}\right), 0 \leq i \leq N$. For any $\lambda_1>0, \lambda_2 \geq L^F_q$, and sufficiently small $h$, denote
\begin{equation}
\begin{aligned}
& A_1(h) \coloneqq 2 k^b+\lambda_1+L^\sigma_x+ L^b_x h, \\
& A_2(h) \coloneqq \left(\lambda_1^{-1}+h\right) L^b_p+L^\sigma_p, \\
& A_3(h) \coloneqq -\frac{\ln \left(1-\left(2 k^F+\lambda_2\right) h\right)}{h}, \\
& A_4(h) \coloneqq \frac{L^F_x}{\left(1-\left(2 k^F+\lambda_2\right) h\right) \lambda_2} .
\end{aligned}
\end{equation}

Then, we have, for $0 \leq n \leq N$,
\begin{align}
 E\left\| X_{t_n}^{\pi, 1}-X_{t_n}^{\pi, 2} \right\|^2 &\leq A_2 \sum_{i=0}^{n-1} e^{A_1(n-i-1) h}  E\left\| P_{t_i}^{\pi, 1}-P_{t_i}^{\pi, 2} \right\|^2 h, \label{eq:lemma2:x}\\
 E\left\| P_{t_n}^{\pi, 1}-P_{t_n}^{\pi, 2} \right\|^2 &\leq \begin{aligned}[t]
    &e^{A_3(N-n) h} E\left\| P_{t_N}^{\pi, 1}-P_{t_N}^{\pi, 2} \right\|^2\\
    &+ A_4 \sum_{i=n}^{N-1} e^{A_3(i-n) h} E\left\| X_{t_i}^{\pi, 1}-X_{t_i}^{\pi, 2} \right\|^2 h .\label{eq:lemma2:p}
\end{aligned}
\end{align}

\end{lemma}

\begin{proof}
We remark that even though $A_1, A_2, A_3, A_4$ are all functions of $h$, in order to ease the presentation, we do not make this dependence explicit.
Let us define
\begin{align}
\begin{aligned}[t]
\delta X_i &\coloneqq X_{t_i}^{\pi, 1}-X_{t_i}^{\pi, 2},\\
\delta P_i & \coloneqq P_{t_i}^{\pi, 1}-P_{t_i}^{\pi, 2},\\
\delta b_i & \coloneqq b\left(t_i, X_{t_i}^{\pi, 1}, P_{t_i}^{\pi, 1}\right)-b\left(t_i, X_{t_i}^{\pi, 2}, P_{t_i}^{\pi, 2}\right), \\
\delta \sigma_i & \coloneqq \sigma\left(t_i, X_{t_i}^{\pi, 1}, P_{t_i}^{\pi, 1}\right)-\sigma\left(t_i, X_{t_i}^{\pi, 2}, P_{t_i}^{\pi, 2}\right), \\
\delta F_i &\coloneqq F\left(t_i, X_{t_i}^{\pi, 1}, P_{t_i}^{\pi, 1}, Q_{t_i}^{\pi, 1}\right)-F\left(t_i, X_{t_i}^{\pi, 2}, P_{t_i}^{\pi, 2}, Q_{t_i}^{\pi, 2}\right).
\end{aligned}
\end{align}
Then, we have
\begin{equation}\label{deltaX}
\delta X_{i+1}=\delta X_i+\delta b_i h+\delta \sigma_i \Delta W_i,
\end{equation}
\begin{equation}\label{deltaP}
\delta P_i=E\left( \delta P_{i+1}+\delta F_i h \mid \mathcal{F}_{t_i}\right),
\end{equation}
and motivated by \eqref{formulation3}, we define
\begin{equation}\label{deltaQ}
\delta Q_i\coloneqq \frac{1}{h} E\left(\delta P_{i+1} \Delta W_i^\top \mid \mathcal{F}_{t_i}\right).
\end{equation}
The martingale representation theorem implies the existence of an $\mathcal{F}_t$-adapted square-integrable process $\left\{\delta Q_t\right\}_{t_i \leq t \leq t_{i+1}}$, such that
\begin{equation}\label{martingale_P1}
\delta P_{i+1} = E\left( \delta P_{i+1} \mid \mathcal{F}_{t_i}\right)  +  \int_{t_i}^{t_{i+1}} \delta Q_t \,d W_t,
\end{equation}
which, together with \eqref{deltaP}, gives us
\begin{equation}\label{martingale_P2}
\delta P_{i+1} = \delta P_i-\delta F_i h  +  \int_{t_i}^{t_{i+1}} \delta Q_t \,d W_t.
\end{equation}
From \eqref{deltaX} and \eqref{martingale_P2}, noting that $\delta X_i, \delta P_i, \delta b_i, \delta \sigma_i$, and $\delta F_i$ are all $\mathcal{F}_{t_i}$ measurable, and $E\left[\Delta W_i \mid \mathcal{F}_{t_i}\right]=0$, $E\left(\int_{t_i}^{t_{i+1}} Q_t \,d W_t \mid \mathcal{F}_{t_i}\right)=0$, we have
\begin{align}  
E \left\| \delta X_{i+1} \right\|^2
& = E \left\| \delta X_i+\delta b_i h\right\|^2  + h E\left\|\delta \sigma_i\right\|^2,\\
E\left\| \delta P_{i+1} \right\|^2
&= E\left\| \delta P_i-\delta F_i h\right\|^2 + \int_{t_i}^{t_{i+1}} E\left\| \delta Q_t\right\|^2 \, dt.
\end{align}
We first establish the proclaimed upper bound for the forward diffusion part in \eqref{eq:lemma2:x}.
Using Assumption \ref{assumption:kb_kf} and Remark \ref{lipschitz_constants}, we can apply the root-mean-square and geometric mean inequality (RMS-GM) and derive, for any $\lambda_1>0$,
\begin{equation}
E \left\| \delta X_{i+1} \right\|^2\begin{aligned}[t]
&=\begin{aligned}[t]
    &\left(1+\left(2 k^b+\lambda_1+L^\sigma_x+ L^b_x h\right) h\right) E\left\|\delta X_i\right\|^2\\ & + \left(\left(\lambda_1^{-1}+h\right) L^b_p+L^\sigma_p\right) E\left\|\delta P_i\right\|^2 h, 
\end{aligned}  
\end{aligned}
\end{equation}
where we recall the definitions of $A_1, A_2$.
Notice that $E\left\|\delta X_0\right\|^2=0$, and, thus, by induction, we have that, for $1 \leq n \leq N$,
\begin{equation}
\begin{aligned}
E\left\|\delta X_n\right\|^2 
& \leq A_2 \sum_{i=0}^{n-1} e^{A_1(n-i-1) h} E\left\|\delta P_i\right\|^2 h,
\end{aligned}
\end{equation}
where  in above derivation  we have used the inequality $(1+x)\leq e^x$, $\forall x\in\R$.

In order to show \eqref{eq:lemma2:p}, we use a similar approach and apply the RMS-GM inequality for any $\lambda_2>0$, which yields
\begin{align}\label{lemma2:delta_p_i}
     E\left\|\delta P_{i+1}\right\|^2 \begin{aligned}[t]
&\geq  \begin{aligned}[t]
    &E\left\|\delta P_i\right\|^2+\int_{t_i}^{t_{i+1}} E\left\|\delta Q_t\right\|^2 \, dt  -  2 k^F h E\left\|\delta P_i\right\|^2 \\
& - \left( \lambda_2  E\left\|\delta P_i\right\|^2 + \lambda_2^{-1} \left(L^F_x E\left\|\delta X_i\right\|^2 + L^F_q E\left\|\delta Q_i\right\|^2\right)\right) h.
\end{aligned}
\end{aligned}
\end{align}
To deal with the integral term in the last inequality, we derive the following relation via Ito's isometry, \eqref{martingale_P2} and \eqref{deltaQ},
\begin{equation}
\delta Q_i = \frac{1}{h} E\left( \int_{t_i}^{t_{i+1}} \delta Q_t \, dt \mid \mathcal{F}_{t_i}\right).
\end{equation}
Thereafter, the Jensen and the Cauchy-Schwartz inequalities imply the following lower bound for the integral term, which extends the estimate in \cite{hanlong2020} to a matrix-valued setting,
\begin{equation}\label{eq:lemma2:integral_term}
\begin{aligned}
E\left\|\delta Q_i\right\|^2 h 
& = \sum_{j=1}^d \sum_{k=1}^m  \frac{1}{h} E\left( E\left(\int_{t_i}^{t_{i+1}} \left(\delta Q_t\right)_{j,k}  \, d t \mid \mathcal{F}_{t_i}\right) \right)^2 \\
&\leq \int_{t_{i}}^{t_{i+1}} E\left\| \delta Q_t \right\|^2 \, dt,
\end{aligned}
\end{equation}
where $(\cdot)_{j,k}$ denotes the $(j,k)$-entry of the matrix. 
Combining \eqref{lemma2:delta_p_i} with \eqref{eq:lemma2:integral_term}, subsequently gives
\begin{equation}\label{deltaP_2}
E\left\|\delta P_{i+1}\right\|^2 \geq \begin{aligned}[t]
    &\left(1-\left(2 k^F +\lambda_2\right) h\right) E\left\|\delta P_i\right\|^2\\
    &+\left(1-L^F_q \lambda_2^{-1}\right) E\left\|\delta Q_i\right\|^2 h - L^F_x \lambda_2^{-1} E\left\|\delta X_i\right\|^2 h .
\end{aligned}
\end{equation}
For any $\lambda_2 \geq L^F_q$ and sufficiently small $h$ satisfying $\left(2 k^F+\lambda_2\right) h<1$, we get the following inequality
\begin{equation}
E\left\|\delta P_i\right\|^2 \leq \left(1-\left(2 k^F+\lambda_2\right) h\right)^{-1} \left(E\left\|\delta P_{i+1}\right\|^2 + L^F_x \lambda_2^{-1} E\left\|\delta X_i\right\|^2 h\right) .
\end{equation}
Finally, recalling the definitions of $A_3, A_4$, by induction, we obtain that, for $0 \leq n \leq N-1$,
\begin{equation}
E\left\|\delta P_n\right\|^2 \leq e^{A_3(N-n) h} E\left\|\delta P_N\right\|^2 + A_4 \sum_{i=n}^{N-1} e^{A_3(i-n) h} E\left\|\delta X_i\right\|^2 h.
\end{equation}
\end{proof}
With the help of the a-priori result above, we can now state the main result of this section, which establishes the \emph{a-posteriori estimate} for the convergence of the deep BSDE algorithm for the SMP formulation of the stochastic control problem under the aforementioned assumptions.
\begin{theorem}\label{thm:estimate2}
 Suppose the conditions for Theorem \ref{convergence_im} hold true , and there exist $\lambda_1>$ $0, \lambda_2 \geq L^F_q, \lambda_3>0$ and constants in $\mathscr{L}$ such that $\overline{A_0}<1$, where
\begin{align}
\overline{A_1}&:=\lim_{h\to 0} A_1(h)=2 k^b+\lambda_1+L^\sigma_x, \\
\overline{A_2}&:=\lim_{h\to 0} A_2(h)=L^b_p \lambda_1^{-1}+L^\sigma_p , \\
\overline{A_3}&:=\lim_{h\to 0} A_3(h)=2 k^F+\lambda_2, \\
\overline{A_4}&:=\lim_{h\to 0} A_4(h)= L^F_x \lambda_2^{-1} , \\
\overline{A_0}&:=\overline{A_2} \frac{1-e^{-\left(\overline{A_1}+\overline{A_3}\right) T}}{\overline{A_1}+\overline{A_3}}\left\{L^{g_{x}}_x e^{\left(\overline{A_1}+\overline{A_3}\right) T}+\overline{A_4} \frac{e^{\left(\overline{A_1}+\overline{A_3}\right) T}-1}{\overline{A_1}+\overline{A_3}}\right\}.\label{eq:thm4:A0bar}
\end{align}
Then, there exists a constant $C>0$, depending on $E\|x_0\|^2, \mathscr{L}, T, \lambda_1$, $\lambda_2$ and $\lambda_3$ such that, for sufficiently small $h$,
\begin{equation}\label{eq:thm4:bound}
\begin{aligned}[b]
    \sup_{t \in[0, T]} 
\Big(E\left\|X_t-\hat{X}_t^\pi\right\|^2 
+ E\left\|P_t-\hat{P}_t^\pi\right\|^2 \Big)\\
+ \int_0^T E\left\|Q_t-\hat{Q}_t^\pi\right\|^2 \, dt 
\end{aligned}
\leq C \big( h + E\left\|-\nabla_x g\left(X_T^\pi\right)-P_T^\pi\right\|^2 \big),
\end{equation}
where $\hat{X}_t^\pi=X_{t_i}^\pi, \hat{P}_t^\pi=P_{t_i}^\pi, \hat{Q}_t^\pi=Q_{t_i}^\pi$ for $t \in\left[t_i, t_{i+1}\right)$, and $\left( X_t, P_t, Q_t \right)$ is the solution to \eqref{eq:FBSDE}.
\end{theorem}

\begin{proof}

Let $X_{t_i}^{\pi, 1}=X_{t_i}^\pi, P_{t_i}^{\pi, 1}=P_{t_i}^\pi, Q_{t_i}^{\pi, 1}=$ $Q_{t_i}^\pi$, given by the Euler scheme \eqref{eq:euler}, and $X_{t_i}^{\pi, 2}=\bar{X}_{t_i}^\pi, P_{t_i}^{\pi, 2}=\bar{P}_{t_i}^\pi, Q_{t_i}^{\pi, 2}=\bar{Q}_{t_i}^\pi$, given by the implicit scheme \eqref{implicit_scheme}. Since both of these schemes solve \eqref{formulation3}, we can apply Lemma \ref{lem:estimate1}. In what follows, we adopt the notation of Lemma \ref{lem:estimate1}.

Through the same reasoning in \cite{hanlong2020}, we first apply the RMS-GM inequality for any $\lambda_3>0$,
\begin{equation}\label{eq:thm4:deltaP_N}
E\left\|\delta P_N\right\|^2 
= E\left\| -\nabla_x g\left(\bar{X}_T^\pi\right)-P_T^\pi\right\|^2 
\leq \begin{aligned}[t]
    & \left(1+\lambda_3^{-1}\right) E\left\| -\nabla_x g\left(X_T^\pi\right)-P_T^\pi\right\|^2\\ 
    &+ L^{g_{x}}_x \left(1+\lambda_3\right) E\left\|\delta X_N\right\|^2 ,
\end{aligned}
\end{equation}
where we also used that $\nabla_x g$ is Lipschitz continuous.

Let 
\begin{equation}\label{eq:thm4:def:LR}
\mathcal{X} :=\max_{0 \leq n \leq N} e^{-A_1 n h} E\left\|\delta X_n\right\|^2, \quad 
\mathcal{P} :=\max_{0 \leq n \leq N} e^{A_3 n h} E\left\|\delta P_n\right\|^2 .
\end{equation}
From the upper bound in \eqref{eq:lemma2:x}, it follows that
\begin{equation}\label{eq:thm4:deltaXn}
e^{-A_1 n h} E\left\|\delta X_n\right\|^2 \leq A_2 \sum_{i=0}^{n-1} e^{-A_1(i+1) h} E\left\|\delta P_i\right\|^2 h \leq A_2 \mathcal{P} \sum_{i=0}^{n-1} e^{-A_1(i+1) h-A_3 i h} h.
\end{equation}
Similarly, \eqref{eq:lemma2:p} and \eqref{eq:thm4:def:LR} imply the following estimate,  extending the scalar case in \cite{hanlong2020} to the vector-valued setting, 
\begin{equation}\label{eq:thm4:deltaPn}
\begin{aligned}[t]
e^{A_3 n h} E\left\|\delta P_n\right\|^2&\leq  e^{A_3 T} E\left\|\delta P_N\right\|^2+A_4 \sum_{i=n}^{N-1} e^{A_3 i h} E\left\|\delta X_i\right\|^2 h \\
&\leq \begin{aligned}[t]
    &e^{A_3 T} \left(1+\lambda_3^{-1}\right) E\left\| - \nabla_x g\left(X_T^\pi\right)-P_T^\pi\right\|^2\\ 
    &+ \left( L^{g_{x}}_x\left(1+\lambda_3\right) e^{\left(A_1+A_3\right) T}+A_4 \sum_{i=n}^{N-1} e^{\left(A_1+A_3\right) i h} h\right)  \mathcal{X},
\end{aligned}
\end{aligned}
\end{equation}
where we used \eqref{eq:thm4:deltaP_N} to obtain the second inequality.
Combining \eqref{eq:thm4:deltaPn}, \eqref{eq:thm4:deltaXn} with \eqref{eq:thm4:def:LR} thus yields
\begin{align}
\mathcal{X} &\leq A_2 h e^{-A_1 h} \frac{e^{-\left(A_1+A_3\right) T}-1}{e^{-\left(A_1+A_3\right) h}-1} \mathcal{P},\label{eq:thm4:bound:L} \\
\mathcal{P} &\leq \begin{aligned}[t]
    &e^{A_3 T}\left(1+\lambda_3^{-1}\right) E\left\| - \nabla_x g\left(X_T^\pi\right)-P_T^\pi\right\|^2\\
    &+  \left(L^{g_{x}}_x\left(1+\lambda_3\right) e^{\left(A_1+A_3\right) T}+A_4 h \frac{e^{\left(A_1+A_3\right) T}-1}{e^{\left(A_1+A_3\right) h}-1}\right) \mathcal{X} .\label{eq:thm4:bound:R}
\end{aligned}
\end{align}
Now, we define
\begin{equation}\label{eq:thm4:def:A}
A(h) \coloneqq A_2 h e^{-A_1 h} \frac{e^{-\left(A_1+A_3\right) T}-1}{e^{-\left(A_1+A_3\right) h}-1} \left(L^{g_{x}}_x\left(1+\lambda_3\right) e^{\left(A_1+A_3\right) T}+A_4 h \frac{e^{\left(A_1+A_3\right) T}-1}{e^{\left(A_1+A_3\right) h}-1}\right),
\end{equation}
by which \eqref{eq:thm4:bound:R} can be rewritten, as follows
\begin{equation}
\mathcal{P} \leq e^{A_3 T}\left(1+\lambda_3^{-1}\right) E\left\|-\nabla_x g\left(X_T^\pi\right)-P_T^\pi\right\|^2 + A(h) \mathcal{P}.
\end{equation}
Hence, whenever $A(h)<1$, the estimates in \eqref{eq:thm4:bound:L} and \eqref{eq:thm4:bound:R} take the following form
\begin{align}
\mathcal{X} &\leq  \frac{e^{A_3 T} \left(1+\lambda_3^{-1}\right) A_2 h e^{-A_1 h} \frac{e^{-\left(A_1+A_3\right) T}-1}{e^{-\left(A_1+A_3\right) h}-1} E\left\| -\nabla_x g\left(X_T^\pi\right)-P_T^\pi\right\|^2}{1-A(h)}, \\
\mathcal{P} &\leq  \frac{e^{A_3 T}\left(1+\lambda_3^{-1}\right) E\left\|-\nabla_x g\left(X_T^\pi\right)-P_T^\pi\right\|^2}{1-A(h)}.
\end{align}
Recall $\lim_{h\to 0} A_i(h)=\overline{A_i}$. From \eqref{eq:thm4:def:A}, it then follows that
\begin{equation}\label{eq:thm4:lim_h A}
\lim_{h \rightarrow 0} A(h)=\overline{A_2} \frac{1-e^{-\left(\overline{A_1}+\overline{A_3}\right) T}}{\overline{A_1}+\overline{A_3}}
\left( L^{g_{x}}_x\left(1+\lambda_3\right) e^{\left(\overline{A_1}+\overline{A_3}\right) T}+\overline{A_4} \frac{e^{\left(\overline{A_1}+\overline{A_3}\right) T}-1}{\overline{A_1}+\overline{A_3}}\right) .
\end{equation}
Recall the definition of $\overline{A_0}$ in \eqref{eq:thm4:A0bar}. Whenever $\overline{A_0}<1$, comparing \eqref{eq:thm4:lim_h A} with $\overline{A_0}$, there exists a sufficiently small $\lambda_3>0$ such that $\lim_{h \rightarrow 0} A(h)<1$ and therefore $A(h)<1$ holds for sufficiently small $h$.
Then we can write, for sufficiently small $h$ and any $\epsilon>0$,
\begin{align}
\mathcal{X} &\leq(1+\epsilon) \frac{\overline{A_2} e^{\overline{A_3} T}}{1-\lim_{h \rightarrow 0} A(h)} \left(1+\lambda_3^{-1}\right) \frac{1-e^{-\left(\overline{A_1}+\overline{A_3}\right) T}}{\overline{A_1}+\overline{A_3}} E\left\|-\nabla_x g\left(X_T^\pi\right)-P_T^\pi\right\|^2,\\
\mathcal{P} &\leq(1+\epsilon)\frac{e^{\overline{A_3} T}}{1-\lim_{h\rightarrow 0} A(h)} \left(1+\lambda_3^{-1}\right) E\left\|-\nabla_x g\left(X_T^\pi\right)-P_T^\pi\right\|^2.
\end{align}
 Note that we have a slight difference in these two estimates compared to \cite{hanlong2020}, where we keep $\lambda_3$ on the right-hand side as it has a significant impact on the term $(1+\lambda_3^{-1})$ and should be chosen appropriately. 
Consequently, by fixing $\epsilon$ and choosing the corresponding small $h>0$, we obtain the following error estimates
\begin{align}
\max_{0 \leq n \leq N} E\left\|\delta X_n\right\|^2 &\leq e^{ A_1 T \vee 0} \mathcal{X} \leq C\left(\lambda_1, \lambda_2, \lambda_3 \right) E\left\|-\nabla_x g\left(X_T^\pi\right)-P_T^\pi\right\|^2, \label{estimate_X}\\
\max_{0 \leq n \leq N} E\left\|\delta P_n\right\|^2 &\leq e^{\left(- A_3 T\right) \vee 0} \mathcal{P} \leq C\left(\lambda_1, \lambda_2, \lambda_3 \right) E\left\|-\nabla_x g\left(X_T^\pi\right)-P_T^\pi\right\|^2 . \\ \label{estimate_P}
\end{align}
Finally, in order to estimate $E\left\|\delta Q_n\right\|^2$ for $0\leq n\leq N-1$, we consider estimate \eqref{deltaP_2} from the proof of Lemma \ref{lem:estimate1}, in which $\lambda_2$ can take any value such that $\lambda_2\geq L^F_q$. In particular, when $L^F_q \neq 0$, we choose $\lambda_2=2 L^F_q$ and obtain
\begin{equation}
\frac{1}{2} E\left\|\delta Q_i\right\|^2 h 
\leq
\frac{L^F_x}{2 L^F_q} E\left\|\delta X_i\right\|^2 h
+ E\left\|\delta P_{i+1}\right\|^2   
- \left(1-\left(2 k^F + 2L^F_q \right) h\right) E\left\|\delta P_i\right\|^2.
\end{equation}
Summing from 0 to $N-1$ therefore gives
\begin{align}
\sum_{i=0}^{N-1} E\left\|\delta Q_i\right\|^2 h 
& \leq \begin{aligned}[t]
    &\frac{L^F_x T}{L^F_q} \max_{0\leq n\leq N} E\left\|\delta X_n\right\|^2\\
    &+ \left( 4\left(k^F+L^F_q\right) T \vee 0 + 2\right) \max_{0\leq n\leq N} E\left\|\delta P_n\right\|^2.
\end{aligned}
\end{align}
Using the estimates established by \eqref{estimate_X} and \eqref{estimate_P}, we collect
\begin{align}\label{estimate_Q}
    \sum_{i=0}^{N-1} E\left\|\delta Q_i\right\|^2 h & \leq C(\lambda_1, \lambda_2, \lambda_3) E\left\|-\nabla_x g\left(X_T^\pi\right)-P_T^\pi\right\|^2.
\end{align}
The case $L^F_q=0$ can be dealt with similarly by choosing $\lambda_2=1$ and the same type of estimate can be derived. 
Finally, combining estimates \eqref{estimate_X}, \eqref{estimate_P} and \eqref{estimate_Q} with Theorem \ref{convergence_im}, we prove our statement.
\end{proof}

\begin{corollary}\label{corollary:control_convergence} Let $\left(X_t, u_t, P_t, Q_t \right)$ be the optimal 4-tuple that solves problem \eqref{stochastic_control_problem}. Under the setting of Theorem \ref{thm:estimate2}, there exists a constant, $C>0$, depending on $E\|x_0\|^2, \mathscr{L}, T, \lambda_1$, and $\lambda_2$, such that for sufficiently small $h$,
\begin{equation}\label{eq:corollary1:bound}
\int_0^T E \| u_t^* - \hat{u}_t^\pi  \|^2 \,dt 
\leq C \big( h + E\left\|-\nabla_x g\left(X_T^\pi\right)-P_T^\pi\right\|^2 \big),
\end{equation}
where $\hat{u}_{t_i}^\pi \coloneqq \mathcal{M}(t, \hat{X}_{t_i}^\pi, \hat{P}_{t_i}^\pi, \hat{Q}_{t_i}^\pi) $, and $\hat{u}_{t}^\pi = \hat{u}_{t_i}^\pi$ for $t \in\left[t_i, t_{i+1}\right)$.
\end{corollary}

\begin{proof}
Recall the optimal feedback map obtained by using the SMP,
\begin{equation}
    u_t = \mathcal{M}(t, X_t, P_t, Q_t), \quad \forall t \in [0, T],
\end{equation}
where $\left(X_t, u_t, P_t, Q_t \right)$ is the optimal 4-tuple that solves stochastic control problem \eqref{stochastic_control_problem} and the FBSDE \eqref{eq:FBSDE}, see Remark \ref{remark:4}. Then, by Theorem \ref{thm:estimate2} and the Lipschitz continuity of $\mathcal{M}$ in Assumption \ref{assumption:holder}, we immediately have
\begin{equation}
\begin{aligned}
\int_0^T E \left\| u_t - \hat{u}_t^\pi  \right\|^2 \,dt
& = \int_0^T E \left\| \mathcal{M}(t, X_t, P_t, Q_t) - \mathcal{M}(t, \hat{X}_{t_i}^\pi, \hat{P}_{t_i}^\pi, \hat{Q}_{t_i}^\pi) \right\|^2 \,dt  \\
& \leq \begin{aligned}[t]
    L \int_0^T \Big(E\left\|X_t-\hat{X}_t^\pi\right\|^2 
&+ E\left\|P_t-\hat{P}_t^\pi\right\|^2 \\
&+ E\left\|Q_t-\hat{Q}_t^\pi\right\|^2 \Big)  \, dt 
\end{aligned}\\
& \leq C \big( h + E\left\|-\nabla_x g\left(X_T^\pi\right)-P_T^\pi\right\|^2 \big).
\end{aligned}
\end{equation}

\end{proof}
 
\section{Numerical Results}\label{sec5}
In this section, we demonstrate the accuracy and robustness of the proposed scheme. We consider stochastic control problems, both in the case of drift and diffusion control. Algorithm \ref{algorithm} has been implemented in TensorFlow 2.15 , and the experiments were run on a Dell Alienware Aurora R10 machine, equipped with an AMD Ryzen 9 3950X CPU (16 cores, $64 \mathrm{Mb}$ cache, $4.7 \mathrm{GHz}$) and an Nvidia GeForce RTX 3090 GPU (24 Gb). The library used in this paper will be publicly accessible under the \href{https://github.com/balintnegyesi}{github} repository of the second author. In all numerical experiments presented below, we use standard fully-connected feedforward neural networks to parametrize $\mu_0^\pi: \mathbb{R}^d \rightarrow \mathbb{R}^d$ and each $\phi_i^{\pi}: \mathbb{R}^d \rightarrow \mathbb{R}^{d\times m}$, $i=0, \dots, N-1$. Each neural network has two hidden layers of width $100$  and a hyperbolic tangent activation function. All parameters are initialized according to default TensorFlow settings.  We employ the Adam optimizer with default settings to minimize the empirical loss function and use a batch size of $2^{12}=4096$  independent paths of the Brownian motion for each SGD step during training.  We apply an adaptive strategy for training the networks. In particular, for $N=2, 5, 10, 20, 50, 100$, we apply an exponential learning rate schedule starting from an initial learning rate $\eta_0^N$, and a decay rate is set to $10^{-6}/\eta_0^N$. An adaptive number of iteration steps $I^N$ also applies to each $N$ \footnote{We set the initial learning rates as \num{5e-4}, \num{5e-4}, \num{1e-3}, \num{2e-3}, \num{4e-3}, \num{8e-3}, \num{8e-3}, and the number of iterations as $2^{12}$, $2^{12}$, $2^{13}$, $2^{14}$, $2^{15}$, $2^{16}$, $2^{17}$, for $N=2, 5, 10, 20, 50, 100$, respectively.
}.  The Monte Carlo errors are computed over an independent validation sample of $2^{14}=16384$  paths. In order to assess the inherent randomness of the deep BSDE methods and the SGD iterations, we run each experiment 5 times and report on the mean and standard deviation of the resulting independent approximations of Algorithm \ref{algorithm}. Computations were carried out with single floating-point precision. We denote the total approximation errors by $\delta \hat{X}^\pi_n$ at time $t_n$ and similarly for other processes. In the comparison with \cite{andersson2023}, we use $\lambda=1$ as described in their paper,  which is also explained in our Section \ref{sec3}.  We denote by $Y_n^\pi$ the discrete time counterpart of the objective functional in \eqref{stochastic_control_problem}, which we compute similarly to \cite[eq. (3.1)]{andersson2023} via a backward summation, using the relations in Remark \ref{PQ_representation}.

Our numerical examples are given for so-called linear quadratic (LQ) type problems, a special case of \eqref{stochastic_control_problem},  where
\begin{align}\label{eq:LQ_coefficients}
& \bar{b}(t,x,u) = Ax + Bu + \beta,\ \bar{\sigma}_{j}(t, x, u) = C_j x + D_j u + \Sigma_j,\\
& \bar{f}(t, x, u) = \frac{1}{2}(x^\top R_x x + u^\top R_{xu} x + u^\top R_u u),\  g(x)= \frac{1}{2} x^\top G x,
\end{align}
with $A\in\R^{d\times d}$, $B\in\R^{d\times \ell}$, $\beta\in\R^{d}$, $R_x \in \mathbb{R}^{d\times d}$, $R_{xu} \in\R^{l\times d}$, $R_u \in \mathbb{R}^{\ell\times \ell}$ and $G \in \mathbb{R}^{d\times d}$, $R_x$, $R_u$ and $G$ are symmetric, and, for $j=1, \dots, m$, $C_j\in\R^{d\times d}$, $D_j\in\R^{d\times \ell}$, $\Sigma_j \in\R^{d}$.
For $\bar{\sigma}, \Sigma,q \in\mathbb{R}^{d\times m}$, we denote the $j$-th column by $\bar{\sigma}_j, \Sigma_j, q_j$. Utilizing the SMP, we obtain the feedback map for the optimal control
\begin{equation}\label{eq:lq:feedback:hamiltonian}
u = \mathcal{M}(t,x,p,q) = - R_u^{-1} \big( R_{xu} x-B^\top p-\sum_{j=1}^m D_j^\top q_{ j} \big), 
\end{equation}
and therein we obtain the corresponding FBSDE \eqref{eq:FBSDE}  for the algorithm,  by substituting \eqref{eq:lq:feedback:hamiltonian} back into the forward diffusion and adjoint equation.

A reference solution to the LQ problem can be derived semi-analytically through numerically solving a system of ODEs, induced by the HJB equation. To obtain this reference solution for a multi-dimensional setting, we extend a well-known result in the literature \cite{yong1999stochastic} for the case $m=1$, and the detailed derivation is given in the appendix.  In what follows, we compare our approximations to a reference solution obtained through the numerical integration of the Riccati ODEs \eqref{eq:ricatti_ode} in the appendix, using an equidistant time grid with $N_\text{ode}=10^5$, and simulating the coupled FBSDE system in \eqref{eq:FBSDE} on the same time grid with a discrete Euler-Maruyama approximation.

\subsection{Example 1 -- drift control}\label{sec:example1}
\begin{figure}
    \centering
    \begin{subfigure}{\textwidth}
        \centering
        \includegraphics[width=1\textwidth]{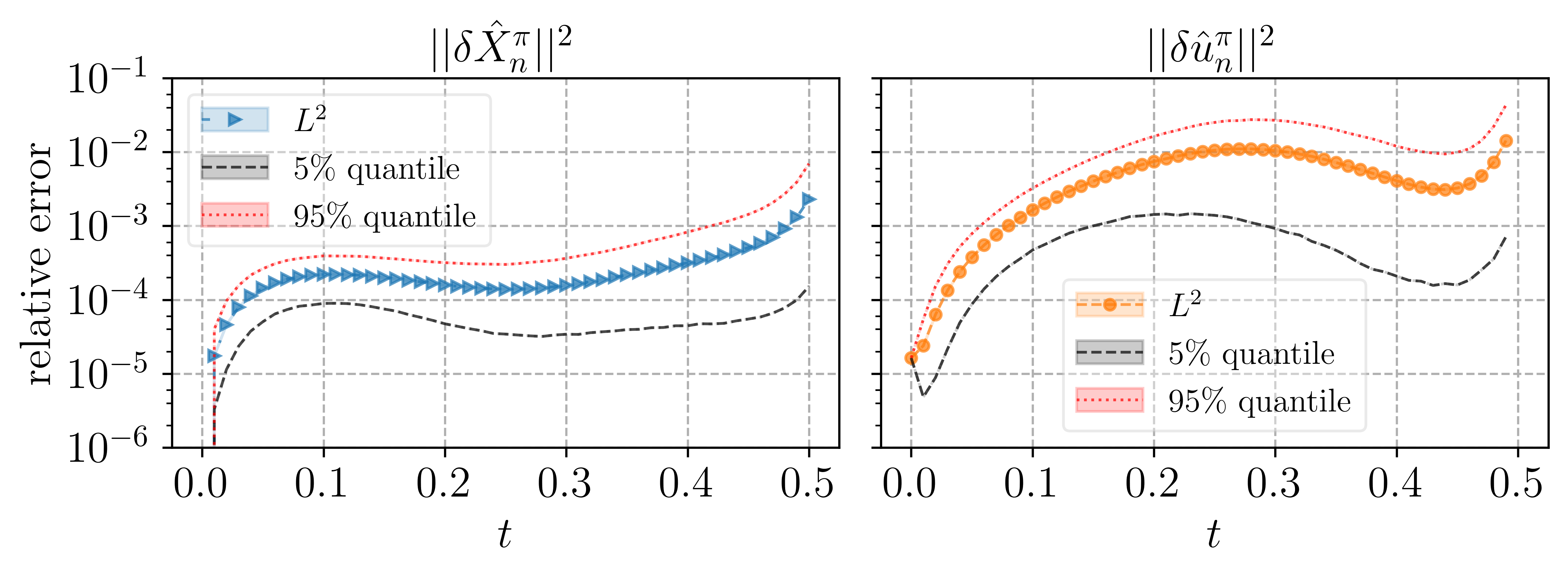}
    \caption{Algorithm \ref{algorithm}.}
    \label{fig:ex1:1:smp}
    \end{subfigure}

    \begin{subfigure}{\textwidth}
        \centering
	\includegraphics[width=1\textwidth]{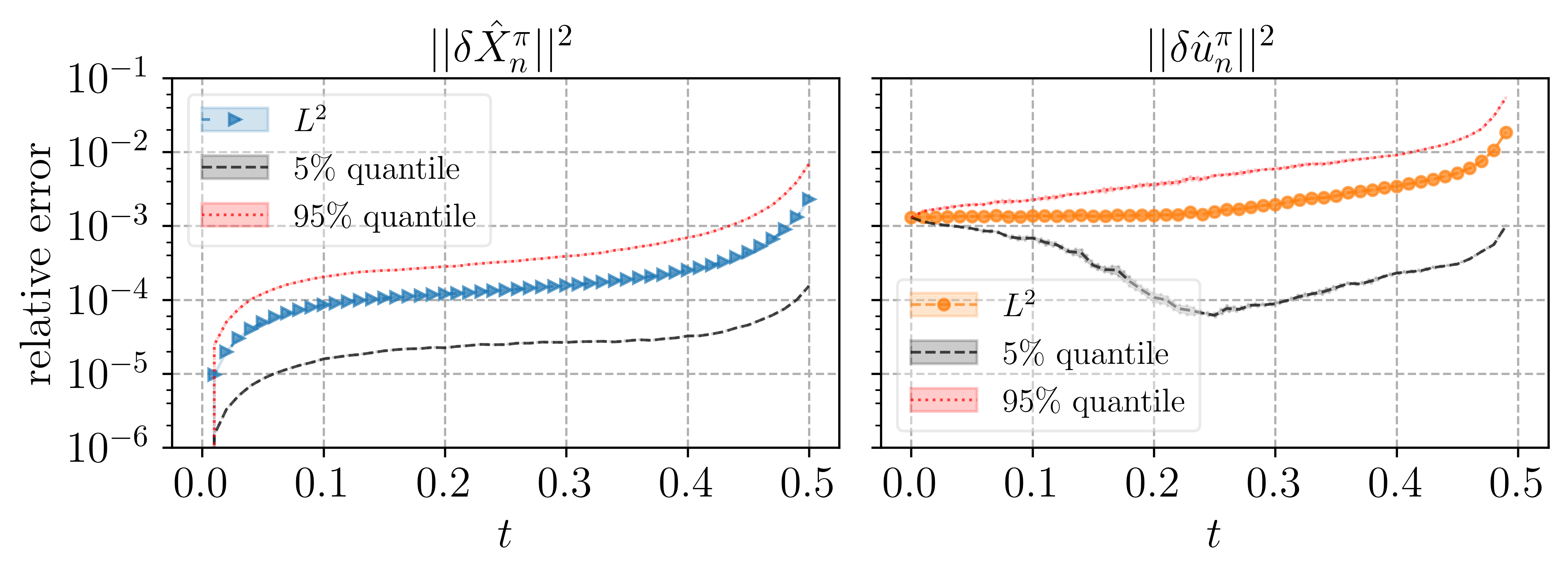}
    \caption{Andersson et al. (2023) in \cite{andersson2023}.}
    \label{fig:ex1:1:andersson}
    \end{subfigure}
    \caption{Example 1, $N=50$. Errors of approximations of the optimally controlled state space and control strategy. On top: results obtained through Algorithm \ref{algorithm} and the SMP. On the bottom: reference method from \cite{andersson2023} using dynamic programming. Lines correspond to the mean of $5$ independent runs of the algorithm, shaded areas to the standard deviation. Graphs computed over an independent Monte Carlo sample of size $M=2^{14}$.}
    \label{fig:ex1:1}
\end{figure}
The following stochastic control problem can be found in \cite[sec. 5.1.2]{andersson2023}. The control space is $\ell=2$ dimensional, whereas for the state space $d=m=6$. The coefficients in \eqref{stochastic_control_problem} are of LQ type and the corresponding matrices in \eqref{eq:LQ_coefficients} are defined, as follows,
\begin{align}\label{eq:coefficients:ex1}
    A &= -\text{diag}([1, 2, 3, 1, 2, 3]),\\
    B &= \begin{pmatrix}
        &1, &1, &0.5, &1, &0, &0\\
        &-1, &1, &1, &-1, &-1, &1
    \end{pmatrix}^\top,\\
    \beta &= -A([-0.2, -0.1, 0, 0, 0.1, 0.2])^\top,\quad C_j=0,\quad D_j=0,\\ \Sigma &= \text{diag}([0.05, 0.25, 0.05, 0.25, 0.05, 0.25]),\quad R_{xu} = 0,\quad
    R_u = 2 I_l,\\
    R_x &= 2 \text{diag}([25, 1, 25, 1, 25, 1]),\quad G=2\text{diag}([1, 25, 1, 25, 1, 25]),
\end{align}
for all $j=1, \dots, m$.
We consider a deterministic initial condition, $x_0=(0.1, \dots, 0.1)$, and a terminal time of $T=1/2$.
We remark that this problem falls under the class of drift control problems, and, recalling the discussion following Assumption \ref{assumption:drift_control}, it in particular satisfies the conditions of Theorem \ref{thm:estimate2} and Corollary \ref{corollary:control_convergence} over any compact subset of the state space.

In Figure \ref{fig:ex1:1}, the relative approximation errors of the optimal control strategy and the optimally controlled state space are compared to the method proposed in \cite{andersson2023},  explained in Section \ref{sec3}.  We remark that, as found in the aforementioned paper, the deep BSDE method on the dynamic programming equation fails to converge to the true solution. As it can be seen our deep BSDE formulation, via the SMP in Algorithm \ref{algorithm}, leads to accurate and robust approximations of both the controlled diffusion and the optimal control.
The relative approximation errors of both processes are $\mathcal{O}(10^{-3})$ for over $95\%$ of the sampling paths. Even though, standard to the FBSDE literature, our theory only establishes error bounds in the $L^2$ sense, see \eqref{eq:thm4:bound} and \eqref{eq:corollary1:bound}, the tight quantile estimates in Figure \ref{fig:ex1:1} suggest that convergence is achieved in a stronger sense as well. Increasing the number of discretization points $N$ further reduces the errors.

The convergence of the proposed algorithm is collected in Figure \ref{fig:conv:ex1}, where the absolute errors of all corresponding processes are plotted for different values of $N$. We recall that the insights of Theorem \ref{thm:estimate2} suggest that these errors admit to an upper bound depending on the \emph{a-posteriori estimate} which is defined as the sum of time step size  $h=T/N$ and the value of the objective functional $E||-\nabla_x g(\hat{X}_N^\pi)-\hat{P}_N^\pi||^2$, see \eqref{eq:thm4:bound}.  Indeed, Figure \ref{fig:conv:ex1} confirms our theoretical findings, where the approximation errors of all processes decay with a rate of at least $\mathcal{O}(h)$. Figure \ref{fig:a_post:ex1} depicts the convergence of the components of the \emph{a-posteriori estimate}. As can be seen from the right end of the curves, once $N$ is large, the optimization problem in Algorithm \ref{algorithm} gets more complex,  and in the case of $N=100$ our optimization strategy only manages to yield a slight decrease for terminal loss, leading to a slower convergence. Note that for a given $N$ whenever the loss functional corresponding to \eqref{euler_objective} is orders of magnitude smaller than the time step size, the complete \emph{a-posteriori estimate} is dominated by the discretization component. Therefore, we still obtain an empirical convergence rate of the \emph{a-posteriori estimate} of  $\mathcal{O}(h^{1.07})$ . In order to preserve convergence for very fine time grids, one needs to make sure the terminal loss is sufficiently small. A good practical guideline, as implied by Theorem \ref{thm:estimate2} and the discussion above, is to ensure that the loss functional's value is significantly smaller than $h$.

The relative $L^2$ approximation errors over time are depicted in Figure \ref{fig:in_time:ex1}, for $N=100$. The method yields highly accurate approximations for all time steps, and of a similar magnitude for all processes. It is worth to mention the accuracy at $t=0$, where no discretization error from the reference solution is present.
All error measures are collected in Table \ref{tab:ex1}. In line with the figures, we see that convergence is achieved in the natural norms of all processes. The resulting approximations are robust regardless of the underlying randomness of the Monte Carlo machinery, standard deviations of the errors are orders of magnitude smaller than their corresponding means. As shown in the last column, Algorithm \ref{algorithm} also offers a competitive runtime for high-dimensional problems, as an SGD iteration step takes approximately  $0.102$   seconds on average when $N=100$. In our experiments, this beats the runtime of \cite{andersson2023} with a factor of $5$  due to the lack of the extra backward summation step considered therein.
\begin{figure}
    \centering
    \begin{subfigure}{0.49\textwidth}
        \centering
        \includegraphics[width=1\textwidth]{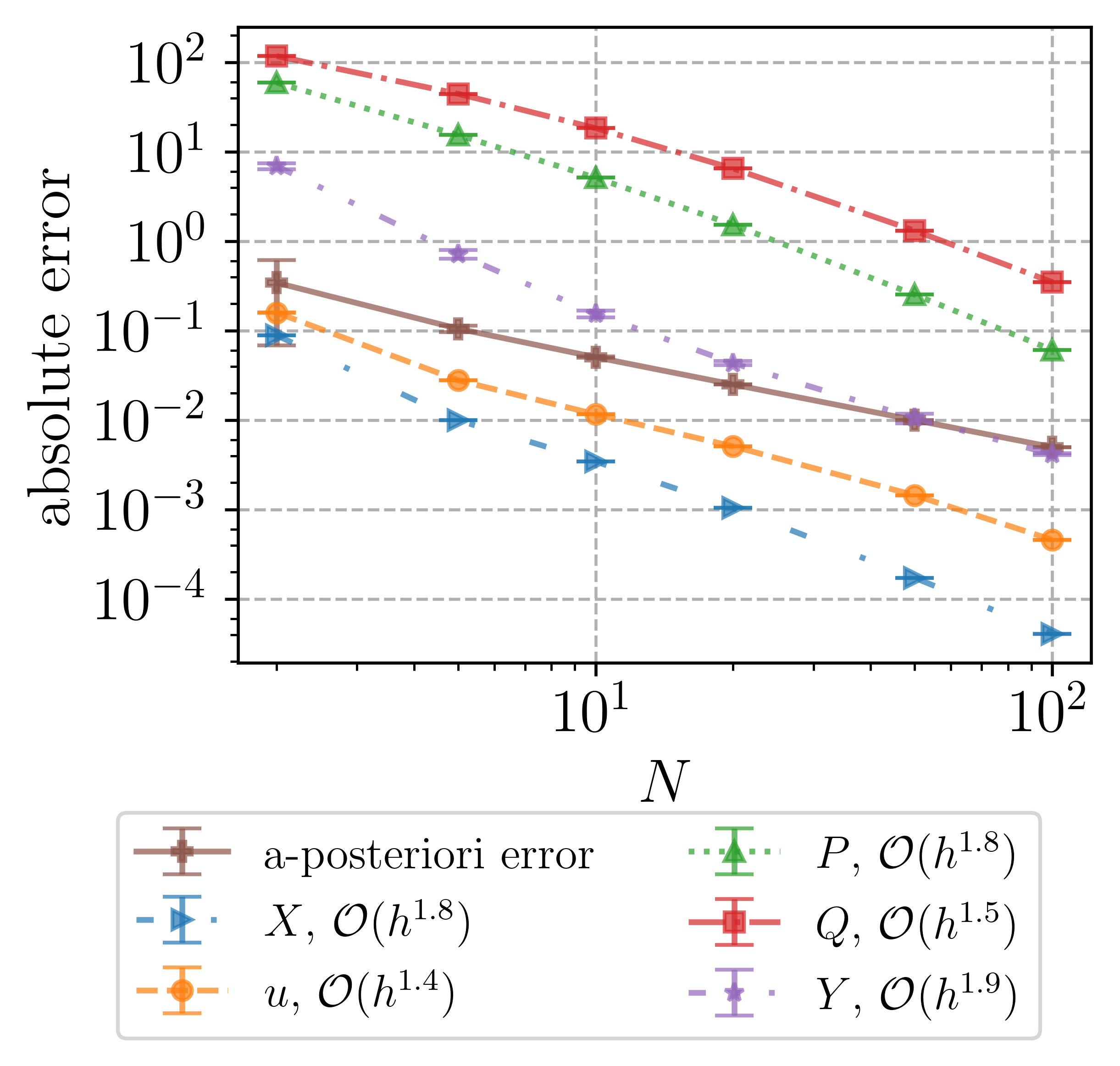}
    \caption{Example 1.}
    \label{fig:conv:ex1}
    \end{subfigure}
    \begin{subfigure}{0.49\textwidth}
        \centering
        \includegraphics[width=1\textwidth]{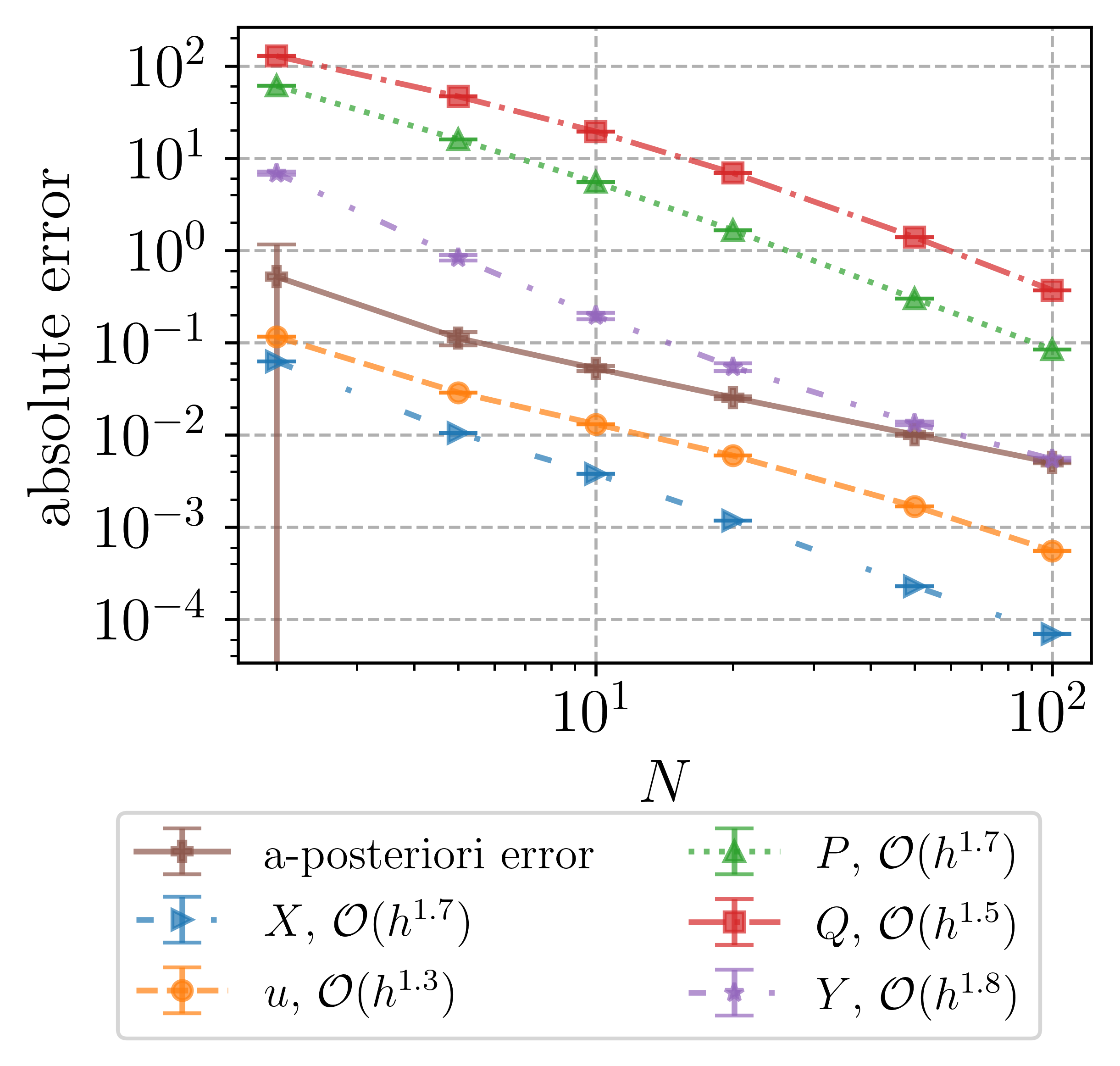}
    \caption{Example 2.}
    \label{fig:conv:ex2}
    \end{subfigure}
    \caption{Convergence and empirical convergence rates of Algorithm \ref{algorithm} over $N$. Lines correspond to the mean of 5 independent runs of the algorithm, error bars to the standard deviation. Graphs compute over an independent Monte Carlo sample of size $M=2^{14}$.}
    \label{fig:conv}
\end{figure}

\begin{figure}
    \centering
    \begin{subfigure}{0.49\textwidth}
        \centering
        \includegraphics[width=1\textwidth]{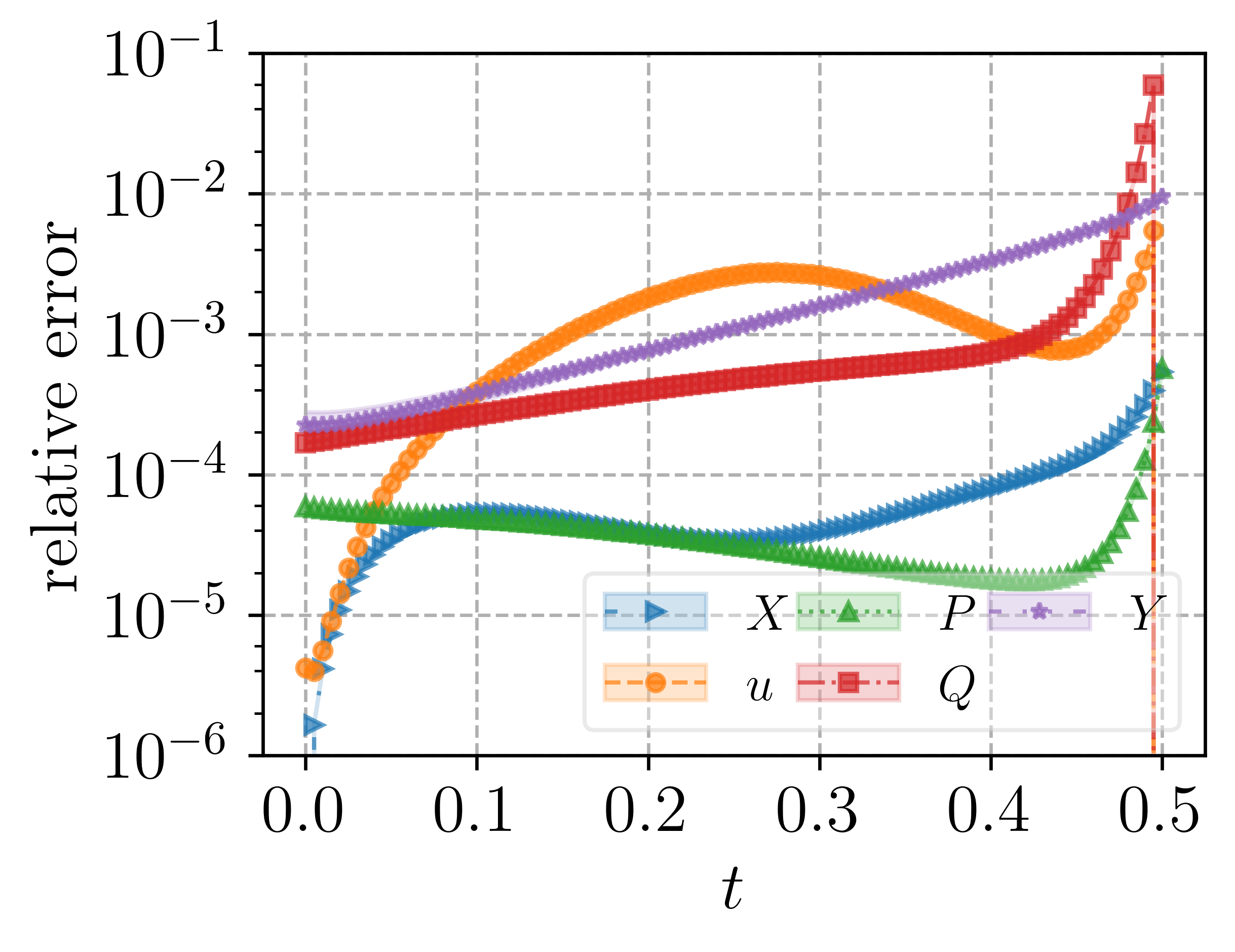}
    \caption{Example 1.}
    \label{fig:in_time:ex1}
    \end{subfigure}
    \begin{subfigure}{0.49\textwidth}
        \centering
        \includegraphics[width=1\textwidth]{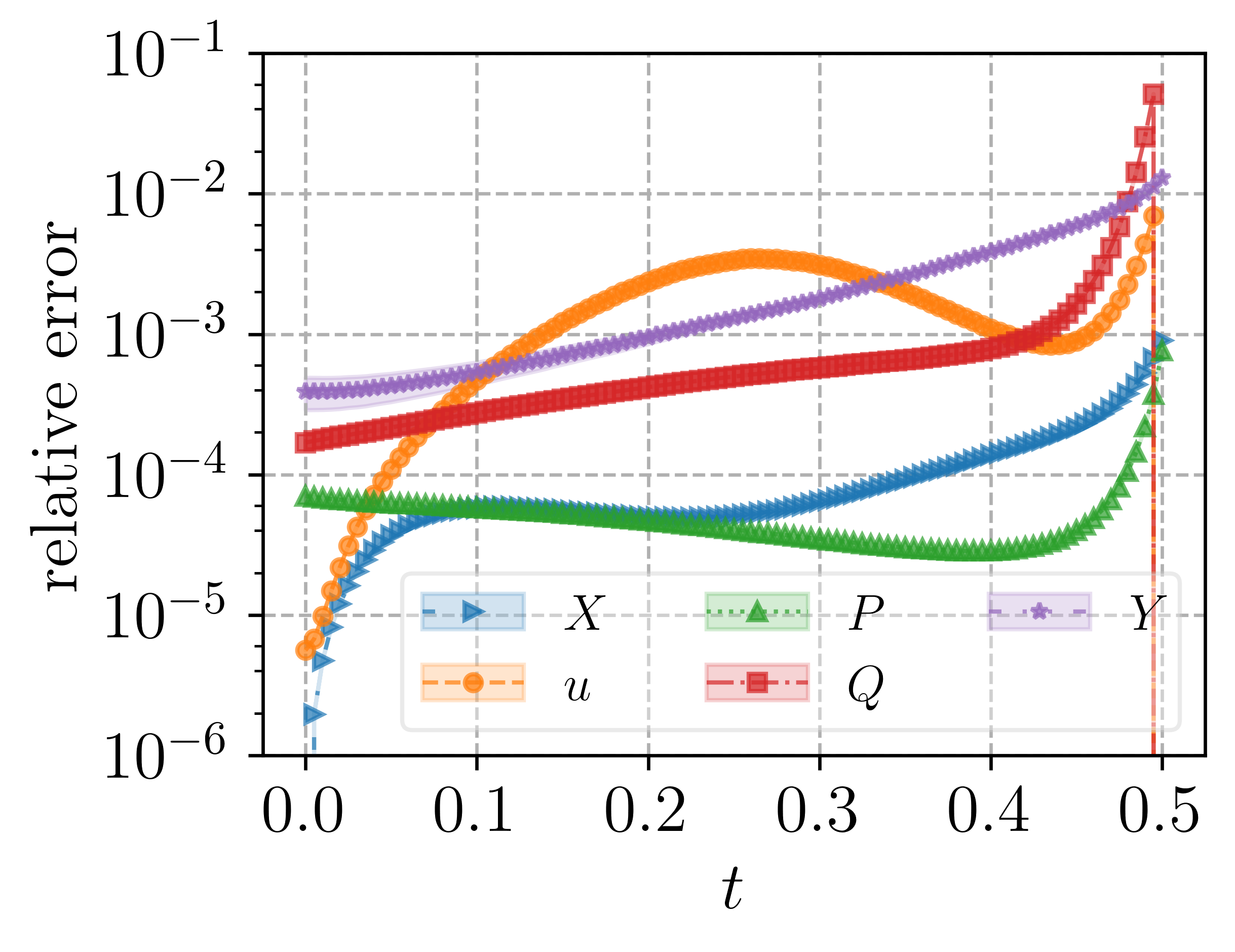}
    \caption{Example 2.}
    \label{fig:in_time:ex2}
    \end{subfigure}
    \caption{Relative $L^2$ approximation errors over time, $N=100$. Lines correspond to the mean of 5 independent runs of the algorithm, shaded areas to the standard deviation. Graphs computed over an independent Monte Carlo sample of size $M=2^{14}$.}
    \label{fig:in_time}
\end{figure}
\begin{figure}
    \centering
    \begin{subfigure}{0.49\textwidth}
        \centering
        \includegraphics[width=1\textwidth]{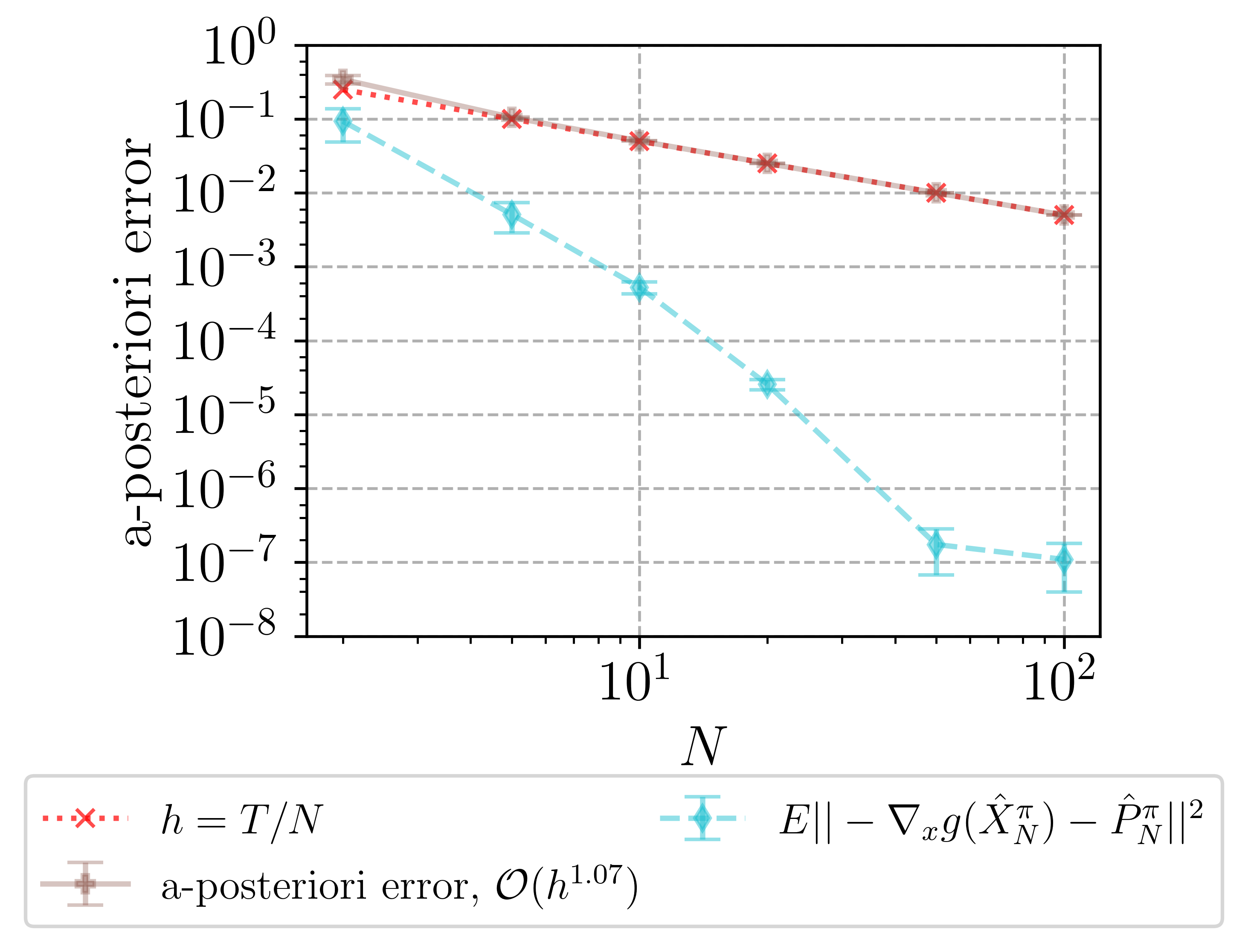}
    \caption{Example 1.}
    \label{fig:a_post:ex1}
    \end{subfigure}
    \begin{subfigure}{0.49\textwidth}
        \centering
        \includegraphics[width=1\textwidth]{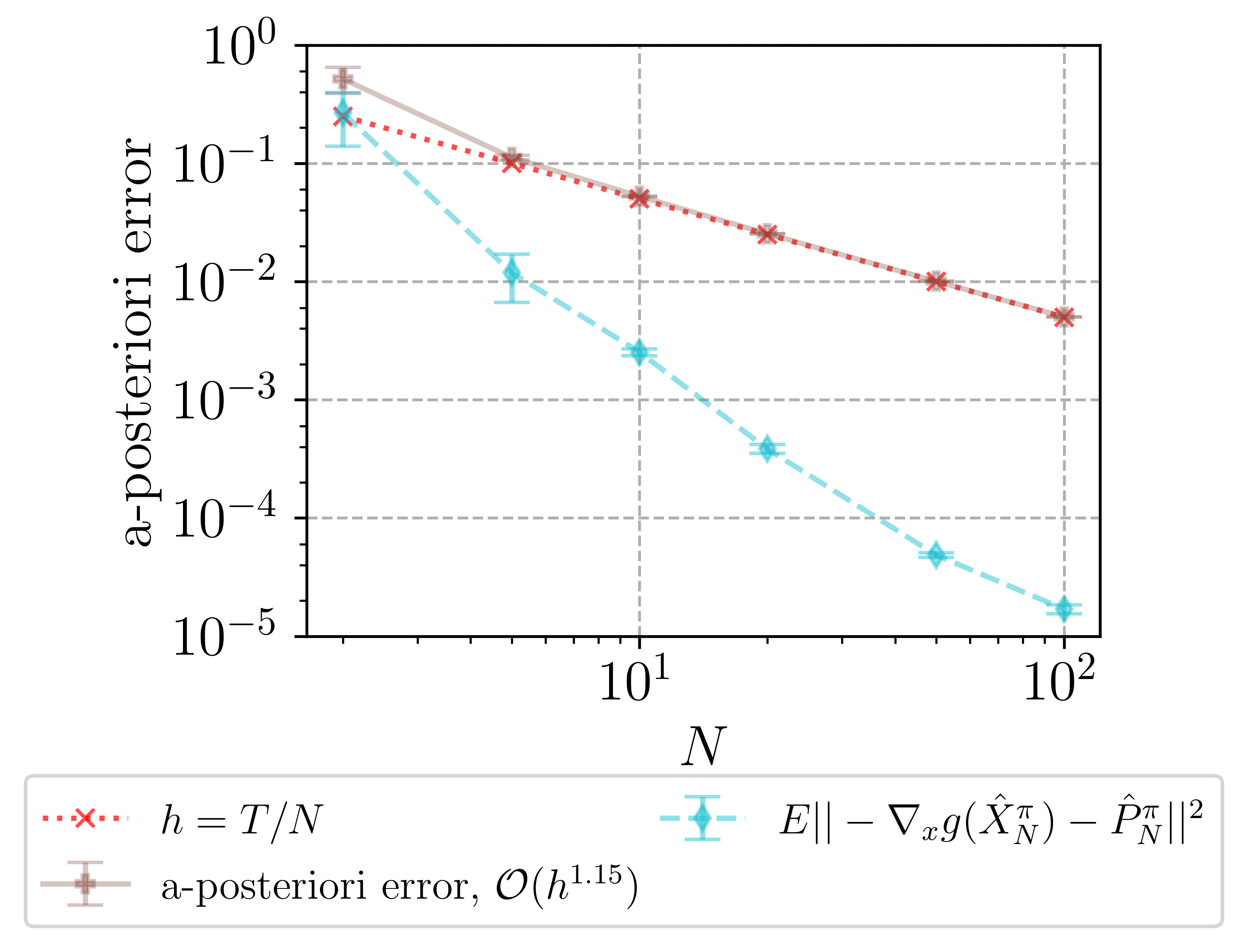}
    \caption{Example 2.}
    \label{fig:a_post:ex2}
    \end{subfigure}
    \caption{Convergence of the a-posteriori error estimate defined in \eqref{eq:thm4:bound}. Lines correspond to the mean of 5 independent runs of the algorithm, error bars to the standard deviation. Graphs compute over an independent Monte Carlo sample of size $M=2^{14}$.}
    \label{fig:a_post}
\end{figure}
\begin{figure}
    \centering
    \includegraphics[width=\textwidth]{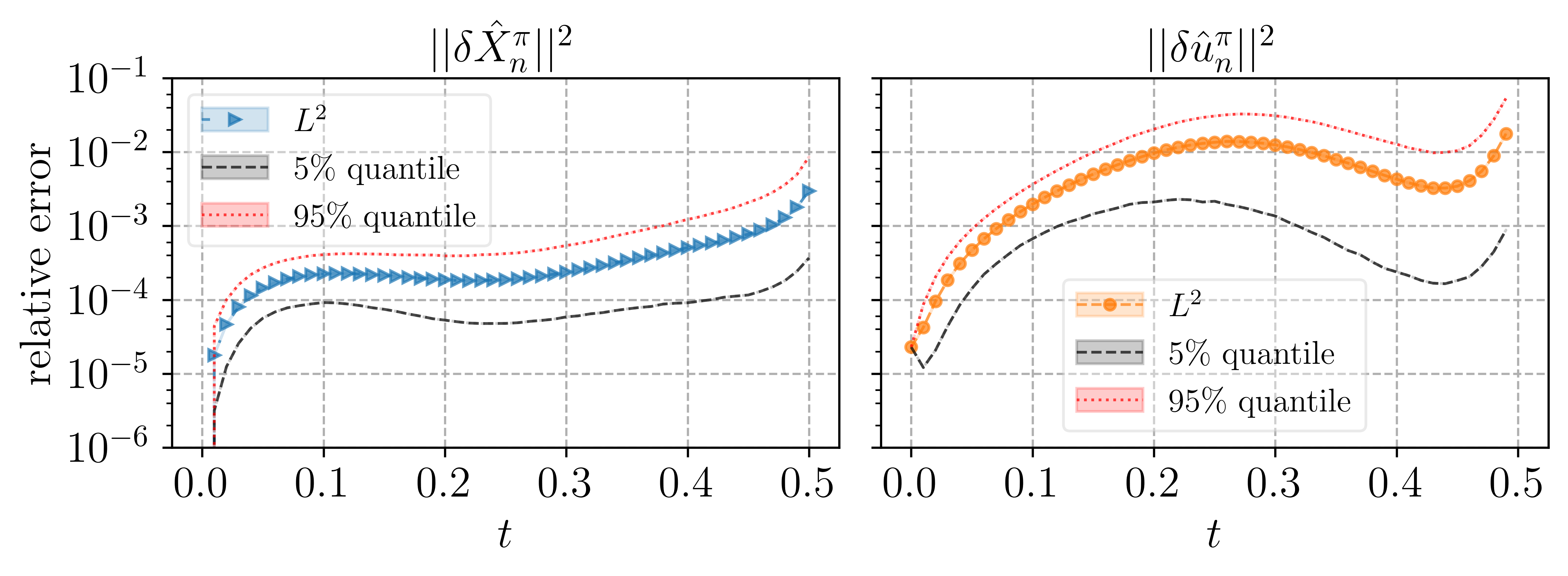}
    \caption{Example 2, $N=50$. Errors of the approximations of the optimally controlled state space and control strategy through Algorithm \ref{algorithm} and the SMP. Lines correspond to the mean of $5$ independent runs of the algorithm, shaded areas to the standard deviation. Graphs computed over an independent Monte Carlo sample of size $M=2^{14}$.}
    \label{fig:ex2}
\end{figure}

\begin{sidewaystable}
\centering
\caption{Example 1. Error measures for different values of $N$. Figures correspond to the mean (standard deviation) of $5$ independent runs of the algorithm. Expectations computed over an independent Monte Carlo sample of size $M=2^{14}$.}\label{tab:ex1}

\resizebox{\textheight}{!}{\begin{tabular}{llllllll}
\toprule
 & max. $E||\delta \hat{X}_n^\pi||^2$ & max. $E||\delta \hat{P}_n^\pi||^2$ & avg. $E||\delta \hat{Q}_n^\pi||^2$ & avg. $E||\delta \hat{u}_n^\pi||^2$ & $E||-\nabla_x g(\hat{X}_N^\pi)-\hat{P}_N^\pi||^2$ & $|\delta \hat{Y}_0^\pi|$ & iter. time (s) \\
\midrule
2 & $\num{8.869e-02}(\num{4e-05})$ & $\num{5.919e+01}(\num{5e-02})$ & $\num{1.1746e+02}(\num{8e-02})$ & $\num{1.5929e-01}(\num{2e-05})$ & $\num{9e-02}(\num{4e-02})$ & $\num{9.1e+00}(\num{6e-01})$ & $\num{4.8e-03}(\num{3e-04})$ \\
5 & $\num{9.9527e-03}(\num{6e-07})$ & $\num{1.53842e+01}(\num{6e-04})$ & $\num{4.4459e+01}(\num{2e-03})$ & $\num{2.7842e-02}(\num{4e-06})$ & $\num{5e-03}(\num{2e-03})$ & $\num{7.0e-01}(\num{8e-02})$ & $\num{7.1e-03}(\num{1e-04})$ \\
10 & $\num{3.44241e-03}(\num{9e-08})$ & $\num{5.1267e+00}(\num{2e-04})$ & $\num{1.83573e+01}(\num{3e-04})$ & $\num{1.1609e-02}(\num{2e-06})$ & $\num{5.3e-04}(\num{1e-04})$ & $\num{1.2e-01}(\num{1e-02})$ & $\num{1.18e-02}(\num{2e-04})$ \\
20 & $\num{1.04521e-03}(\num{2e-08})$ & $\num{1.52850e+00}(\num{3e-05})$ & $\num{6.49980e+00}(\num{9e-05})$ & $\num{5.092e-03}(\num{2e-06})$ & $\num{2.6e-05}(\num{4e-06})$ & $\num{2.5e-02}(\num{3e-03})$ & $\num{2.5e-02}(\num{2e-03})$ \\
50 & $\num{1.72512e-04}(\num{1e-09})$ & $\num{2.52543e-01}(\num{6e-06})$ & $\num{1.31291e+00}(\num{2e-05})$ & $\num{1.4437e-03}(\num{3e-07})$ & $\num{2e-07}(\num{1e-07})$ & $\num{3e-03}(\num{1e-03})$ & $\num{5.72e-02}(\num{4e-04})$ \\
100 & $\num{4.06471e-05}(\num{4e-10})$ & $\num{6.10016e-02}(\num{5e-07})$ & $\num{3.49517e-01}(\num{7e-06})$ & $\num{4.5906e-04}(\num{8e-08})$ & $\num{1.1e-07}(\num{7e-08})$ & $\num{5e-04}(\num{1e-04})$ & $\num{1.019e-01}(\num{2e-04})$ \\
\bottomrule
\end{tabular}
}
\vspace{4cm}
\caption{Example 2. Error measures for different values of $N$. Figures correspond to the mean (standard deviation) of $5$ independent runs of the algorithm. Expectations computed over an independent Monte Carlo sample of size $M=2^{14}$.}\label{tab:ex2}
\resizebox{\textheight}{!}{\begin{tabular}{llllllll}
\toprule
 & max. $E||\delta \hat{X}_n^\pi||^2$ & max. $E||\delta \hat{P}_n^\pi||^2$ & avg. $E||\delta \hat{Q}_n^\pi||^2$ & avg. $E||\delta \hat{u}_n^\pi||^2$ & $E||-\nabla_x g(\hat{X}_N^\pi)-\hat{P}_N^\pi||^2$ & $|\delta \hat{Y}_0^\pi|$ & iter. time (s) \\
\midrule
2 & $\num{6.242e-02}(\num{5e-05})$ & $\num{6.065e+01}(\num{7e-02})$ & $\num{1.277e+02}(\num{1e-01})$ & $\num{1.1582e-01}(\num{8e-05})$ & $\num{3e-01}(\num{1e-01})$ & $\num{8.3e+00}(\num{3e-01})$ & $\num{4.7e-03}(\num{4e-04})$ \\
5 & $\num{1.04819e-02}(\num{2e-07})$ & $\num{1.5949e+01}(\num{3e-03})$ & $\num{4.666e+01}(\num{1e-02})$ & $\num{2.880e-02}(\num{3e-05})$ & $\num{1.2e-02}(\num{5e-03})$ & $\num{7.8e-01}(\num{6e-02})$ & $\num{7.1e-03}(\num{1e-04})$ \\
10 & $\num{3.7835e-03}(\num{1e-07})$ & $\num{5.4819e+00}(\num{5e-04})$ & $\num{1.94561e+01}(\num{1e-03})$ & $\num{1.3102e-02}(\num{2e-06})$ & $\num{2.5e-03}(\num{2e-04})$ & $\num{1.5e-01}(\num{2e-02})$ & $\num{1.16e-02}(\num{1e-04})$ \\
20 & $\num{1.18283e-03}(\num{7e-08})$ & $\num{1.6518e+00}(\num{1e-04})$ & $\num{6.8986e+00}(\num{4e-04})$ & $\num{5.9651e-03}(\num{7e-07})$ & $\num{3.8e-04}(\num{3e-05})$ & $\num{3.2e-02}(\num{6e-03})$ & $\num{2.5e-02}(\num{2e-03})$ \\
50 & $\num{2.2904e-04}(\num{2e-08})$ & $\num{2.9989e-01}(\num{2e-05})$ & $\num{1.39010e+00}(\num{9e-05})$ & $\num{1.6851e-03}(\num{2e-07})$ & $\num{4.9e-05}(\num{2e-06})$ & $\num{4.1e-03}(\num{7e-04})$ & $\num{5.75e-02}(\num{5e-04})$ \\
100 & $\num{6.9501e-05}(\num{4e-09})$ & $\num{8.4388e-02}(\num{9e-06})$ & $\num{3.7029e-01}(\num{5e-05})$ & $\num{5.5453e-04}(\num{9e-08})$ & $\num{1.7e-05}(\num{1e-06})$ & $\num{9e-04}(\num{2e-04})$ & $\num{1.023e-01}(\num{3e-04})$ \\
\bottomrule
\end{tabular}
}

\end{sidewaystable}

\subsection{Example 2 -- drift and diffusion control}\label{sec:example2}
In the next section, we provide an extension to the previous problem by introducing diffusion control to the LQ problem in Section \ref{sec:example1}. The coefficients read,
\begin{align}\label{eq:coefficients:ex2}
    C_j &= \frac{1}{60}\text{diag}([1, 2, 3, 1, 2, 3]),\quad 
    D_j &= \frac{1}{60}\begin{pmatrix}
        &1, &0, &1, &0, &1, &0\\
        &0, &-1, &0, &-1, &0, &-1
    \end{pmatrix}^T,
\end{align}
for all $j=1, \dots, m$, and the rest as in \eqref{eq:coefficients:ex1}. Note that due to the presence of diffusion control, this equation can not be treated via standard DP, such as \cite{andersson2023}, therefore no reference method is provided.

In Figure \ref{fig:ex2}, the relative approximation errors of the controlled state process and the optimal control strategy are depicted. Even with the introduction of diffusion control, we obtain accurate approximations up to $\mathcal{O}(10^{-2})$ relative accuracy in both processes. The tight quantile bounds suggest that our approximations remain accurate even in stronger senses than $L^2$.

As can easily be seen from \eqref{eq:lq:feedback:hamiltonian}, $Q$ couples back into the forward diffusion and hence the conditions of Assumption \ref{assumption:drift_control} are not satisfied. Therefore, this example does not fall under the setting of Theorem \ref{thm:estimate2}. Nevertheless, the convergence drawn in Figure \ref{fig:conv:ex2},  together with a related theoretical result Corollary 4.7 in \cite{reisinger2023}, gives hope that a similar theoretical bound using the weak and monotonicity conditions may also be established in the diffusion control case.  Figure \ref{fig:a_post:ex2} collects the convergence of the \emph{a-posteriori estimate}. Similar to the drift control case we find that the \emph{a-posteriori estimate} is dominated by the discretization component over the considered time steps $N$, and therefore \emph{a-posteriori estimate} converges with a rate of  $\mathcal{O}(h^{1.15})$.  These results are comparable to Example 1.

The relative $L^2$ approximation errors are depicted in Figure \ref{fig:in_time:ex2} for $N=100$. As we see, even in the case of diffusion control, we preserve a similar relative approximation error of $\mathcal{O}(10^{-3})$, as in Example 1. The slight increase in the relative error of $u$ is due to the feedback map \eqref{eq:lq:feedback:hamiltonian}, also accumulating the errors in $Q$. Finally, all error measures are collected in Table \ref{tab:ex2},  which shows that all processes converge as expected. Moreover,  the lack of difference in total runtime in Table \ref{tab:ex2}, compared to the drift control case in \ref{tab:ex1}, highlights the great potential of such deep BSDE formulations in the framework of high-dimensional diffusion control problems.

\section*{Appendix}

 In this appendix, we follow the same approach used for $m=1$ in \cite{yong1999stochastic}, and derive the reference solution for general dimensions $d$, $\ell$ and $m$.  By the DP, the value function $V(t, x)$ for a LQ problem should satisfy the HJB equation \eqref{eq:HJB} with a boundary condition $V(T, x) = \frac{1}{2} x^\top G x$, and, in the LQ case, the generalized Hamiltonian $\mathcal{G}$ is given by 
\begin{equation}
\begin{aligned}
\mathcal{G}\left(t, x, u ,-V_x, -V_{xx}\right)
= & -\frac{1}{2} u^\top R_u u  -   u^\top R_{xu} x - \frac{1}{2} x^\top R_x x - (Ax + B u + \beta)^\top V_x \\
& - \sum_{j=1}^m \frac{1}{2} (C_jx + D_ju + \Sigma_j)^\top V_{xx}(C_jx + D_ju + \Sigma_j).
\end{aligned}
\end{equation}
We conjecture that the value function takes the form $V(t, x)= \frac{1}{2} x^\top \Gamma(t) x + x^\top \gamma(t) + \kappa(t)$, for some unknowns symmetric $\Gamma(t)\in \mathbb{R}^{d\times d}$, $\gamma(t)\in\mathbb{R}^d$ and $\kappa(t)\in\mathbb{R}$, and for simplicity we omit the argument $t$ in the following. Then
\begin{align}
\mathcal{G}(t, x, u, -V_x, -V_{x x}) =
\begin{aligned}[t]  & -\frac{1}{2}  (u+\Psi x+\psi)^\top \widehat{R} (u+\Psi x+\psi) 
+ \frac{1}{2} x^\top  \widehat{S}^{\top} \widehat{R}^{-1} \widehat{S} x \\
&+ \frac{1}{2} x^\top  \left( - R_x - \sum_{j=1}^m C_j^{\top} \Gamma  C_j - \Gamma^{\top} A-A^{\top} \Gamma \right)  x  \\
&  - x^\top \left( A^{\top} \gamma + \sum_{j=1}^m C_j^{\top} \Gamma \Sigma_j -\Psi^{\top} \widehat{R} \psi + \Gamma \beta \right)\\
&+ \frac{1}{2}  \psi^\top \widehat{R} \psi  - \beta^\top \gamma - \sum_{j=1}^m \frac{1}{2} \Sigma_j^\top \Gamma \Sigma_j,
\end{aligned}
\end{align}
where we define the following
\begin{equation}
\begin{aligned}
& \widehat{R} \coloneqq R_u + \sum_{j=1}^m D_j^{\top} \Gamma D_j, \quad \widehat{S} \coloneqq B^{\top} \Gamma + R_{xu} + \sum_{j=1}^m D_j^{\top} \Gamma C_j, \\
& \Psi \coloneqq \widehat{R}^{-1} \widehat{S}, \quad \psi \coloneqq\widehat{R}^{-1}\left(B^{\top} \gamma + \sum_{j=1}^m D_j^{\top} \Gamma \Sigma_j \right),
\end{aligned}
\end{equation}
provided that $\widehat{R}$ is positive definite.
Due to the first quadratic term, we immediately see that the optimal control should take the following feedback form,
\begin{equation}\label{eq:ricatti_ode}
u^*
= - \Big(R_u + \sum_{j=1}^m D_j^{\top} V_{xx} D_j \Big)^{-1}  \begin{aligned}[t]
    \Big(&B^{\top} V_x + \big(R_{xu} + \sum_{j=1}^m D_j^{\top} V_{xx} C_j\big) x\\
    &+   \sum_{j=1}^m D_j^{\top} V_{xx} \Sigma_j\Big).
\end{aligned}
\end{equation}
Substituting $\mathcal{G}(t, x, u^*, -V_x, -V_{xx})$ back to \eqref{eq:HJB}, we find that $V(t, x)= \frac{1}{2} x^\top \Gamma(t) x + x^\top \gamma(t) + \kappa(t)$ solves the HJB equation whenever the following relations hold, at each $t\in [0, T]$
\begin{equation}\label{eq:Ric}
\left\{  
\begin{aligned}
& 0 = \dot{\Gamma} + \Gamma A+A^{\top} \Gamma + \sum_{j=1}^m C_j^{\top} \Gamma C_j + R_x - \widehat{S}^{\top} \widehat{R}^{-1} \widehat{S}, \\
& 0 = \dot{\gamma}+ A^{\top} \gamma + \sum_{j=1}^m C_j^{\top} \Gamma \Sigma_j -\Psi^{\top} \widehat{R} \psi + \Gamma \beta,\\
& 0 = \dot{\kappa} - \frac{1}{2} \psi^\top \widehat{R} \psi + \beta^\top\gamma + \sum_{j=1}^m \frac{1}{2} \Sigma_j^\top \Gamma \Sigma_j,\\
& \Gamma(T)=G,\quad \gamma(T)=0,\quad \kappa(T)=0. 
\end{aligned}
\right.
\end{equation}
This system of ODEs can be integrated numerically with practically arbitrary accuracy.
With the thereby obtained numerical solution of \eqref{eq:Ric}, one can subsequently compute the value function and all of its derivatives 
using the conjecture, and obtain the optimal control $u^*_t$ through the feedback map \eqref{eq:ricatti_ode}. The reference solution to the corresponding BSDE \eqref{eq:FBSDE}, derived via the SMP, can similarly be computed using the relations \eqref{PQ_representation_drift} in Remark \ref{PQ_representation}.

\section*{Acknowledgement}
The first author would like to thank the China Scholarship Council (CSC) for the financial support. The second author acknowledges financial support from the Peter Paul Peterich Foundation via the TU Delft University Fund. The authors would also like to thank the anonymous referees for their valuable comments and suggestions for improving the paper.

\bibliographystyle{unsrt} % We choose the "plain" reference style
\bibliography{ref.bib} % Entries are in the refs.bib file

\begin{thebibliography}{10}

\bibitem{ji2022}
Shaolin Ji, Shige Peng, Ying Peng, and Xichuan Zhang.
\newblock Solving stochastic optimal control problem via stochastic maximum principle with deep learning method.
\newblock {\em Journal of Scientific Computing}, 93(1):30, 2022.

\bibitem{hanlong2020}
Jiequn Han and Jihao Long.
\newblock Convergence of the deep {BSDE} method for coupled {FBSDEs}.
\newblock {\em Probability, Uncertainty and Quantitative Risk}, 5:1--33, 2020.

\bibitem{bellman1958dynamic}
Richard Bellman.
\newblock Dynamic programming and stochastic control processes.
\newblock {\em Information and control}, 1(3):228--239, 1958.

\bibitem{bismut1978introductory}
Jean-Michel Bismut.
\newblock An introductory approach to duality in optimal stochastic control.
\newblock {\em SIAM review}, 20(1):62--78, 1978.

\bibitem{pontryagin2018mathematical}
Lev~Semenovich Pontryagin.
\newblock {\em Mathematical theory of optimal processes}.
\newblock Routledge, 2018.

\bibitem{kushner1990numerical}
Harold~J Kushner.
\newblock Numerical methods for stochastic control problems in continuous time.
\newblock {\em SIAM Journal on Control and Optimization}, 28(5):999--1048, 1990.

\bibitem{krylov2004rate}
Nicolai~V. Krylov.
\newblock The {Rate} of {Convergence} of {Finite}-{Difference} {Approximations} for {Bellman} {Equations} with {Lipschitz} {Coefficients}.
\newblock {\em Applied Mathematics and Optimization}, 52(3):365--399, October 2005.

\bibitem{dong2007rate}
Hongjie Dong and Nicolai~V Krylov.
\newblock The rate of convergence of finite-difference approximations for parabolic bellman equations with lipschitz coefficients in cylindrical domains.
\newblock {\em Applied Mathematics and Optimization}, 56:37--66, 2007.

\bibitem{jakobsen2003rate}
Espen~Robstad Jakobsen.
\newblock On the rate of convergence of approximation schemes for {Bellman} equations associated with optimal stopping time problems.
\newblock {\em Mathematical Models and Methods in Applied Sciences}, 13(05):613--644, 2003.

\bibitem{han2016}
Jiequn Han and Weinan E.
\newblock Deep learning approximation for stochastic control problems.
\newblock {\em arXiv preprint arXiv:1611.07422}, 2016.

\bibitem{hunt1992neural}
Kenneth~J Hunt, D~Sbarbaro, R~{\.Z}bikowski, and Peter~J Gawthrop.
\newblock Neural networks for control systems—a survey.
\newblock {\em Automatica}, 28(6):1083--1112, 1992.

\bibitem{lehalle1998piecewise}
Charles-Albert Lehalle and Robert Azencott.
\newblock Piecewise affine neural networks and nonlinear control.
\newblock In {\em ICANN 98: Proceedings of the 8th International Conference on Artificial Neural Networks, Sk{\"o}vde, Sweden, 2--4 September 1998 8}, pages 633--638. Springer, 1998.

\bibitem{psaltis1988}
D.~Psaltis, A.~Sideris, and A.A. Yamamura.
\newblock A multilayered neural network controller.
\newblock {\em IEEE Control Systems Magazine}, 8(2):17--21, 1988.

\bibitem{bachouch2022}
Achref Bachouch, C{\^o}me Hur{\'e}, Nicolas Langren{\'e}, and Huyen Pham.
\newblock Deep neural networks algorithms for stochastic control problems on finite horizon: numerical applications.
\newblock {\em Methodology and Computing in Applied Probability}, 24(1):143--178, 2022.

\bibitem{hure2021deep}
C{\^o}me Hur{\'e}, Huy{\^e}n Pham, Achref Bachouch, and Nicolas Langren{\'e}.
\newblock Deep neural networks algorithms for stochastic control problems on finite horizon: convergence analysis.
\newblock {\em SIAM Journal on Numerical Analysis}, 59(1):525--557, 2021.

\bibitem{pereira2020feynman}
Marcus Pereira, Ziyi Wang, Tianrong Chen, Emily Reed, and Evangelos Theodorou.
\newblock {Feynman-Kac} neural network architectures for stochastic control using second-order {FBSDE} theory.
\newblock In {\em Learning for Dynamics and Control}, pages 728--738. PMLR, 2020.

\bibitem{han2018}
Jiequn Han, Arnulf Jentzen, and Weinan E.
\newblock Solving high-dimensional partial differential equations using deep learning.
\newblock {\em Proceedings of the National Academy of Sciences}, 115(34):8505--8510, 2018.

\bibitem{fouque2020deep}
Jean-Pierre Fouque and Zhaoyu Zhang.
\newblock Deep learning methods for mean field control problems with delay.
\newblock {\em Frontiers in Applied Mathematics and Statistics}, 6:11, 2020.

\bibitem{carmona2022convergence}
Ren{\'e} Carmona and Mathieu Lauri{\`e}re.
\newblock Convergence analysis of machine learning algorithms for the numerical solution of mean field control and games: {II}—the finite horizon case.
\newblock {\em The Annals of Applied Probability}, 32(6):4065--4105, 2022.

\bibitem{hu2023recent}
Ruimeng Hu and Mathieu Lauri{\`e}re.
\newblock Recent developments in machine learning methods for stochastic control and games.
\newblock {\em arXiv preprint arXiv:2303.10257}, 2023.

\bibitem{germain2021neural}
Maximilien Germain, Huy{\^e}n Pham, Xavier Warin, et~al.
\newblock Neural networks-based algorithms for stochastic control and {PDEs} in finance.
\newblock {\em arXiv preprint arXiv:2101.08068}, 2021.

\bibitem{andersson2023}
Kristoffer Andersson, Adam Andersson, and Cornelis~W Oosterlee.
\newblock Convergence of a robust deep {FBSDE} method for stochastic control.
\newblock {\em SIAM Journal on Scientific Computing}, 45(1):A226--A255, 2023.

\bibitem{Jiang2021}
Yifan Jiang and Jinfeng Li.
\newblock Convergence of the deep bsde method for fbsdes with non-lipschitz coefficients.
\newblock {\em Probability, Uncertainty and Quantitative Risk}, 6(4):391--408, 2021.

\bibitem{reisinger2023}
Christoph Reisinger, Wolfgang Stockinger, and Yufei Zhang.
\newblock {A posteriori error estimates for fully coupled McKean–Vlasov forward-backward SDEs}.
\newblock {\em IMA Journal of Numerical Analysis}, page drad060, 09 2023.

\bibitem{bender2013}
Christian Bender and Jessica Steiner.
\newblock A posteriori estimates for backward sdes.
\newblock {\em SIAM/ASA Journal on Uncertainty Quantification}, 1(1):139--163, 2013.

\bibitem{hure2020deep}
C{\^o}me Hur{\'e}, Huy{\^e}n Pham, and Xavier Warin.
\newblock Deep backward schemes for high-dimensional nonlinear {PDEs}.
\newblock {\em Mathematics of Computation}, 89(324):1547--1579, 2020.

\bibitem{negyesi2021one}
Balint Negyesi, Kristoffer Andersson, and Cornelis~W Oosterlee.
\newblock {The One Step Malliavin scheme: new discretization of BSDEs implemented with deep learning regressions}.
\newblock {\em IMA Journal of Numerical Analysis}, page drad092, 02 2024.

\bibitem{gao2023convergence}
Chengfan Gao, Siping Gao, Ruimeng Hu, and Zimu Zhu.
\newblock Convergence of the backward deep {BSDE} method with applications to optimal stopping problems.
\newblock {\em SIAM J. Financial Mathematics}, 14(4):1290--1303, 2023.

\bibitem{pham2009continuous}
Huy{\^e}n Pham.
\newblock {\em Continuous-time stochastic control and optimization with financial applications}, volume~61.
\newblock Springer Science \& Business Media, 2009.

\bibitem{yong1999stochastic}
Jiongmin Yong and Xun~Yu Zhou.
\newblock {\em Stochastic controls: Hamiltonian systems and HJB equations}, volume~43.
\newblock Springer Science \& Business Media, 1999.

\bibitem{bender2008time}
Christian Bender and Jianfeng Zhang.
\newblock Time discretization and {Markovian} iteration for coupled {FBSDEs}.
\newblock {\em The Annals of Applied Probability}, 18(1):143--177, February 2008.

\bibitem{BENSOUSSAN2015}
A.~Bensoussan, S.C.P. Yam, and Z.~Zhang.
\newblock Well-posedness of mean-field type forward–backward stochastic differential equations.
\newblock {\em Stochastic Processes and their Applications}, 125(9):3327--3354, 2015.

\bibitem{peng1999}
Shige Peng and Zhen Wu.
\newblock Fully coupled forward-backward stochastic differential equations and applications to optimal control.
\newblock {\em SIAM Journal on Control and Optimization}, 37(3):825--843, 1999.

\bibitem{antonelli1993backward}
Fabio Antonelli.
\newblock {\em Backward forward stochastic differential equations}.
\newblock Purdue University, 1993.

\bibitem{pardoux1999forward}
Etienne Pardoux and Shanjian Tang.
\newblock Forward-backward stochastic differential equations and quasilinear parabolic pdes.
\newblock {\em Probability theory and related fields}, 114:123--150, 1999.

\bibitem{zhang2017backward}
Jianfeng Zhang.
\newblock {\em Backward stochastic differential equations}.
\newblock Springer, 2017.

\end{thebibliography}




% \end{thebibliography}
\end{document}